\documentclass[11pt,twoside]{amsart}
\usepackage{amsmath,amssymb,amsthm}
\usepackage{cite}

\newtheorem{thm}{Theorem}[section]

 \newtheorem{cor}{Corollary}[section]
 \newtheorem{lem}{Lemma}[section]
 \newtheorem{prop}{Proposition}[section]
 \newtheorem{defn}{Definition}[section]
\newtheorem{rem}{Remark}[section]

\def\d{\partial}

\def\tilde{\widetilde}

\begin{document}
\title[Non-isentropic compressible Navier-Stokes system]{The large-time behavior of solutions in the critical $L^p$ framework for compressible viscous and heat-conductive gas flows}

\author{Weixuan Shi}
\address{Department of Mathematics, Nanjing
University of Aeronautics and Astronautics,
Nanjing 211106, P.R.China}

\address{Department of Mathematics and Statistics, McGill University, Montreal, Quebec H3A 2K6, Canada}
\email{wxshi168@163.com}

\author{Jiang Xu}
\address{Department of Mathematics, Nanjing
University of Aeronautics and Astronautics,
Nanjing 211106, P.R.China}
\email{jiangxu\underline{ }79math@yahoo.com}


\begin{abstract}
The $L^{p}$ theory for non-isentropic Navier-Stokes equations governing compressible viscous and heat-conductive gases is not yet proved completely so far, because the critical regularity cannot control all non linear coupling terms. In this paper, we pose an additional regularity assumption of low frequencies in $\mathbb{R}^d(d\geq 3)$, and then the sharp time-weighted inequality can be established, which leads to the time-decay estimates of global strong solutions in the $L^{p}$ critical Besov spaces. Precisely, we show that if the initial data belong to some Besov space $\dot{B}^{-s_{1}}_{2,\infty}$ with $s_{1}\in (1-\frac{d}{2}, s_0](s_0\triangleq \frac{2d}{p}-\frac{d}{2})$, then the $L^{p}$ norm of the critical global solutions admits the time decay $t^{-\frac{s_{1}}{2}-\frac{d}{2}(\frac{1}{2}-\frac{1}{p})}$ (in particular, $t^{-\frac{d}{2p}}$ if $s_1=s_0$), which coincides with that of heat kernel in the $L^p$ framework. In comparison with \cite{DX2}, the low-frequency regularity $s_1$ can be improved to be \textit{the whole range}. 
\end{abstract}

\maketitle
\section{Introduction}\setcounter{equation}{0}
The compressible viscous and heat conductive gases reads as
\begin{equation} \label{Eq:1.1}
\left\{
\begin{array}{l}
\partial _{t}\varrho +\mathrm{div}\left( \varrho u\right) =0, \\ [2mm]
\partial _{t}( \varrho u) +\mathrm{div}\left( \varrho u\otimes u\right)+\nabla P=\mathrm{div}\,\tau,\\  [2mm]
\partial_{t}\left[\varrho\left(\frac{|u|^{2}}{2}+e\right)\right]+\mathrm{div}\left[u\left(\varrho\left(\frac{|u|^{2}}{2}+e\right)+P\right)\right]
=\mathrm{div}\left(\tau\cdot u-q\right)
\end{array}
\right.
\end{equation}
for $(t,x)\in \mathbb{R}_{+} \times \mathbb{R}^{d}$.
Here, $\varrho=\varrho(t,x)\in \mathbb{R}_{+}$ denotes the density, $u =u (t,x) \in \mathbb{R}^{d}$, the velocity field and $e=e(t,x)\in \mathbb{R}_{+}$, the internal energy per unit mass. We restrict ourselves to the case of a Newtonian fluid: the viscous stress tensor is $\tau=\lambda\,\mathrm{div}\,u \,\mathrm{Id}+2\mu\, D(u)$, where $D(u)\triangleq\frac{1}{2}\left(\nabla u+{}^T\!\nabla u\right)$ stands for the deformation tensor. The notations $\mathrm{div}$ and $\nabla$ are the divergence operator and gradient operator with respect to the spatial variable $x$, respectively. The Lam\'{e} coefficients $\lambda$ and $\mu$ (the \textit{bulk and shear viscosities}) are density-dependent functions, which are supposed to be smooth enough and to satisfy
\begin{equation}\label{Eq:1.2}
\mu>0 \ \ \hbox{and} \ \ \nu\triangleq\lambda +2\mu >0.
\end{equation}
The heat conduction $q$ is given by $q=-\kappa \nabla\mathcal{T}$, where $\mathcal{T}$ stands for the temperature. The heat conduction coefficient $\kappa$ is assumed to be density-dependent smooth function satisfying $\kappa>0$.

It follows from the second and third equations of \eqref{Eq:1.1} that
\begin{equation*}
\partial_{t}\left(\varrho e\right)+\mathrm{div}\left(\varrho ue\right)+P\mathrm{div}\,u=\mathrm{div}\left(\kappa(\varrho) \nabla \mathcal{T}\right)+2\mu D(u):D(u)+\lambda \left(\mathrm{div}\,u\right)^{2}.
\end{equation*}
In order to reformulate \eqref{Eq:1.1} in light of $\varrho$, $u$ and $\mathcal{T}$ only, we make the additional assumption that the internal energy $e=e(\varrho,\mathcal{T})$ satisfies Joule law:
\begin{equation}\label{Eq:1.3}
\partial _{\mathcal{T}}e=C_{v} \ \ \hbox{for some positive constant} \ \ C_{v}
\end{equation}
and that the pressure function $P=P(\varrho,\mathcal{T})$ is of the form
\begin{equation}\label{Eq:1.4}
P(\varrho,\mathcal{T})=\pi_{0}(\varrho)+\mathcal{T}\pi_{1}(\varrho),
\end{equation}
where $\pi_{0}$ and $\pi_{1}$ are given smooth functions. Such pressure laws cover the cases of ideal fluids (for which $\pi_{0}(\varrho)=0$ and $\pi_{1}(\varrho)=R\varrho$ for a universal constant $R>0$), of barotropic fluids ($\pi_{1}(\varrho)=0$), and of Van der Waals fluids ($\pi_{0}=-\alpha\varrho^{2}$, $\pi_{1}=\beta \varrho/(\delta-\varrho)$ with $\alpha,\beta, \delta >0$). With the aid of the Gibbs relations for the internal energy and the Helmholtz free energy, we have the Maxwell relation
\begin{equation*}
\varrho^{2}\partial_{\varrho}e(\varrho,\mathcal{T})=P(\varrho,\mathrm{T})-\mathcal{T}\partial_{\mathcal{T}}P(\varrho,\mathcal{T})=\pi_{0}(\varrho),
\end{equation*}
and end up with the following temperature equation:
\begin{equation} \label{Eq:1.5}
\varrho C_{v}\left(\partial_{t}\mathcal{T}+u\cdot\nabla \mathcal{T}\right)+\mathcal{T}\pi_{1}\left(\varrho\right)\mathrm{div}u=2\mu\, D(u):D(u)+\lambda \left(\mathrm{div}\,u\right)^{2}+\mathrm{div}\left(\kappa (\varrho) \nabla \mathcal{T}\right).
\end{equation}
We focus on solutions that are close to some constant equilibrium $(\varrho_{\infty},0,\mathcal{T}_{\infty})$ with $\varrho_{\infty}>0$ and $\mathcal{T}_{\infty}>0$ fulfilling the linear stability condition:
\begin{equation}\label{Eq:1.6}
\partial_{\varrho}P(\varrho_{\infty},\mathcal{T}_{\infty})>0 \ \ \hbox{and} \ \ \partial_{\mathcal{T}}P(\varrho_{\infty},\mathcal{T}_{\infty})>0.
\end{equation}

If System \eqref{Eq:1.1} is written in terms of  $(\varrho,u,\mathcal{T})$, then it is not difficult to see that \eqref{Eq:1.1} is scaling invariant (neglecting the lower order pressure term) under the following transformation.
\begin{equation} \label{scaling}
\varrho(t,x) \rightsquigarrow \varrho(l^{2}t,lx),\ \ u(t,x) \rightsquigarrow lu(l^{2}t,lx), \ \ \mathcal{T}(t,x) \rightsquigarrow l^{2} \mathcal{T}(l^{2}t,lx), \ \ l>0.
\end{equation}
Consequently, some so-called \textit{critical spaces} was employed to solve \eqref{Eq:1.1}, whose norms are invariant with respect to the scaling.
To the best of our knowledge, the point of view of scaling invariance is now classical and stems from the study of incompressible Navier-Stokes equations, see \cite{CM,FK,KY} and references therein. In comparison with isentropic case (see\cite{CD1,CMZ1,DR1,DR4,DH,DX1,HB,OM,XJ}), the $L^{p}$ theory of \eqref{Eq:1.1} is not completely proved yet. Danchin \cite{DR2} first used general $L^p$ Besov space (chain of spaces $\dot{B}^{d/p}_{p,1}\times\dot{B}^{d/p-1}_{p,1}\times\dot{B}^{d/p-2}_{p,1}$ in fact) and established the local existence and uniqueness of solutions of \eqref{Eq:1.1}. Later, Chikami and Danchin \cite{CD2} performed Lagrangian approach and Banach fixed point theorem to improve those results as in \cite{DR2} such that $1<p<d$ and $d\geq 3$. The exponent $p$ seems to be optimal since the ill-posedness of \eqref{Eq:1.1} in dimension three in the sense that the continuity of data-solution map fails at the origin, was established by Chen, Miao and Zhang\cite{CMZ2} if $p>3$. Danchin \cite{DR3} constructed the global existence and uniqueness of strong solutions to \eqref{Eq:1.1} in the $L^{2}$ critical hybrid Besov spaces (in space dimension $d\geq 3$). Recently, Danchin \& He \cite{DH}
gave the $L^{p}$ extension of \cite{DR3}. For simplicity, those physical coefficients $\lambda$, $\mu$ and $\kappa$ are assumed to be constant. In fact, their results still hold true in case that $\lambda$, $\mu$ and $\kappa$ depend smoothly on the density.

A natural question is what is the large time asymptotic description of the constructed solution in \cite{DH}. For that issue, recall that in the framework of high Sobolev regularity, Matsumura and Nishida \cite{MN1} obtained the fundamental $L^{1}$-$L^{2}$ decay estimate, by assuming the initial data are the small perturbation in $H^{3}(\mathbb{R}^{3})\times L^{1}(\mathbb{R}^{3})$ of $(\varrho_{\infty},0,\mathcal{T}_{\infty})$:
\begin{equation} \label{Eq:1.7}
\|(\varrho-\varrho_{\infty},u,\mathcal{T}-\mathcal{T}_{\infty})(t)\|_{L^{2}(\mathbb{R}^{3})}\lesssim \langle t\rangle^{-\frac{3}{4}}\quad \hbox{with} \quad
\langle t\rangle\triangleq \sqrt{1+t^2}.
\end{equation}

Shortly after Matsumura and Nishida, still for with high Sobolev regularity, there are a number of results on the large-time behavior of solutions to the compressible Navier-Stokes system (also including the present full case), see \cite{HZ, KK1, KK2,KS, KT, LDL, LW, MN2, ZY} and references therein. Precisely, the result of \cite{MN1} was generalized to more physical situations where the fluid domain is not $\mathbb{R}^{d}$: for instance, the exterior domains were studied by Kobayashi \cite{KT} and Kobayashi \& Shibata \cite{KS}, and the half spaces were investigated by Kagei \& Kobayashi \cite{KK1,KK2}. On the other hand, there are some results available which are connected to the wave aspect of the solutions. In one dimension space, Zeng \cite{ZY} presented the $L^{1}$ convergence to the nonlinear Burgers' diffusive wave. Hoff and Zumbrun \cite{HZ} performed the detailed analysis of the Green function for the multi-dimensional case and established the $L^{\infty}$ decay rates of diffusion waves. In \cite{LW}, Liu and Wang gave pointwise convergence of solution to diffusion waves with the optimal time-decay estimate in odd dimension, where the phenomena of the weaker Huygens' principle was also shown. This was generalized later to \eqref{Eq:1.1} in \cite{LDL}. In the critical regularity framework however, there are few results concerning the time-decay estimates of (strong) global solutions to the Cauchy problem of \eqref{Eq:1.1}. Very recently, Danchin and the second author \cite{DX2} made an attempt, where the initial data are additionally assumed to in $\dot{B}^{-s_{1}}_{2,\infty}$ with $\left[\max(0,2-\frac{d}{2}), s_{0}\right](s_{0}\triangleq\frac{2d}{p}-\frac{d}{2})$. Consequently, the $L^{p}$ norm of solutions (the slightly stronger $\dot{B}^{0}_{p,1}$ norm in fact) decays as fast as $t^{-\frac{s_1}{2}}$. In particular, the rate is of  $O(t^{-d(\frac{1}{p}-\frac{1}{4})})$ in case of $s_1=s_0$. However, that is \textit{not optimal} in sense of the decay rate of heat kernel (see Remark \ref{Rem1.3} below).

 %we intend to improve that result of \cite{DX2} such that $s_1$ belongs to \textit{the whole range} $(1-d/2,s_{0}]$ and \textit{the decay rate is the form of} $O(t^{-\frac{s_{1}}{2}-\frac{d}{2}(\frac{1}{2}-\frac{1}{p})})$, which coincides with that of heat kernel in $L^p$ space.

\subsection{Main results}

To simplify the statement, let us assume that the density and the temperature tend to some positive constants $\varrho_{\infty}$ and $\mathcal{T}_{\infty}$, at infinity. Setting $\mathcal{A}\triangleq \mu_{\infty}\Delta+(\lambda_{\infty}
+\mu_{\infty})\nabla\mathrm{div}$, $\varrho=\varrho_{\infty}(1+b)$ and $\mathcal{T}=\mathcal{T}_{\infty}+\mathcal{E}$, we see from \eqref{Eq:1.1} and \eqref{Eq:1.5} that, whenever $b>-1$, the triplet $(b,u,\mathcal{E})$ fulfills
\begin{equation*}
\left\{
\begin{array}{l}
\partial _{t}b+u\cdot\nabla b +(1+b)\mathrm{div} u=0, \\[2mm]
\partial _{t}u+u\cdot\nabla u+\frac{\partial_{\varrho}P\left(\varrho_{\infty}(1+b),\mathcal{T}_{\infty}\right)}{1+b}\nabla b+\frac{\pi_{1}\left(\varrho_{\infty}(1+b)\right)}{\varrho_{\infty}(1+b)}\nabla\mathcal{E}+\frac{\pi'_{1}(\varrho_{\infty}(1+b))}{1+b}\mathcal{E}\nabla b\\[2mm]
\hspace{1.8cm}=\frac{1}{\varrho_{\infty}(1+b)} \mathrm{div}\left(2\mu (\varrho_{\infty}(1+b))D(u)+\lambda (\varrho_{\infty}(1+b))\mathrm{div}\,u \,\mathrm{Id}\right),\\[2mm]
\partial_{t}\mathcal{E}+u\cdot\nabla \mathcal{E}+\left(\mathcal{T}_{\infty}
+\mathcal{E}\right)\frac{\pi_{1}\left(\varrho_{\infty}(1+b)\right)}
{\varrho_{\infty}C_{v}(1+b)}\mathrm{div}\,u=\frac{1}{\varrho_{\infty}C_{v}(1+b)}(\mathrm{div}\left(\kappa(b)\nabla \mathcal{E} \right)\\[2mm]
\hspace{1.8cm}+2\mu(\varrho_{\infty}(1+b)) D(u):D(u)+\lambda(\varrho_{\infty}(1+b))\left(\mathrm{div}\,u\right)^{2}).
\end{array}
\right.
\end{equation*}
Then, setting $\nu\triangleq2\mu_{\infty}+\lambda_{\infty}$ ($\mu_{\infty}=\mu(\varrho_{\infty}) \ and \ \lambda_{\infty}=\lambda(\varrho_{\infty})$), $\nu_{\infty}\triangleq\frac{\nu}{\varrho_{\infty}}$, $\chi_{0}\triangleq \partial_{\varrho }P(\varrho_{\infty},\mathcal{T}_{\infty})^{-\frac{1}{2}}$, and performing the change of unknowns
\begin{equation*}
a(t,x)=b(\nu_{\infty}\chi^{2}_{0}\,t,\nu_{\infty}\chi_{0}\,x),\ \ \ v(t,x)=\chi_{0}u(\nu_{\infty}\chi^{2}_{0}\,t,\nu_{\infty}\chi_{0}\,x),
\end{equation*}
\begin{equation*}
\theta(t,x)=\chi_{0}\sqrt{\frac{C_{v}}{\mathcal{T}_{\infty}}}\mathcal{E}(\nu_{\infty}\chi^{2}_{0}\,t,\nu_{\infty}\chi_{0}\,x),
\end{equation*}
we finally get
\begin{equation}\label{Eq:1.8}
\left\{
\begin{array}{l}
\partial _{t}a+\mathrm{div}\,v=f, \\[2mm]
\partial _{t}v-\tilde{\mathcal{A}}v+\nabla a+\gamma \nabla \theta=g,\\[2mm]
\partial_{t}\theta-\beta\Delta\theta+\gamma\,\mathrm{div}\,v=k
\end{array}
\right.
\end{equation}
with
\begin{equation*}
\tilde{\mathcal{A}}\triangleq\frac{\mathcal{A}}{\nu}, \ \ \beta\triangleq \frac{\kappa_{\infty}}{\nu C_{v}} \  \left(\kappa_{\infty}=\kappa(\varrho_{\infty})\right), \ \ \gamma\triangleq \frac{\chi_{0}}{\varrho_{\infty}}\sqrt{\frac{\mathcal{T}_{\infty}}{C_{v}}}\pi_{1}(\varrho_{\infty}),
\end{equation*}
and where the nonlinear terms $f$, $g$ and $k$ are given by
\begin{eqnarray*}
	f&\triangleq&-\mathrm{div}(av), \\
	g &\triangleq& -v\cdot \nabla v-I(a)\tilde{\mathcal{A}}v-K_{1}(a)\nabla a-K_{2}(a)\nabla \theta-\theta \nabla K_{3}(a)\\
	&&+\frac{1}{(1+a)\nu}\mathrm{div}\left(2\tilde{\mu}(a) D(v)+\tilde{\lambda}(a) \mathrm{div}\, v \,\mathrm{Id}\right),\\
	k &\triangleq&-v \cdot \nabla \theta -\beta I(a) \Delta \theta -\left(\tilde{K}_{1}(a)+\tilde{K}_{2}(a)\,\theta\right) \mathrm{div}\, v+\frac{1}{\nu(1+a)}\mathrm{div}(\tilde{\kappa}(a)\nabla \theta)\\
	&&+\frac{1}{\nu\chi_{0}}\sqrt{\frac{1}{\mathcal{T}_{\infty}C_{v}}}\left(\frac{2\mu\left(\varrho_{\infty}(1+a)\right)}{1+a} D(v): D(v)
	+\frac{\lambda\left(\varrho_{\infty}(1+a)\right)}{1+a}(\mathrm{div}\,v)^{2}\right)
\end{eqnarray*}
with
\begin{eqnarray*}
&&\hspace{-6mm}I(a) \triangleq \frac{a}{1+a}, \ \ \tilde{\mu}(a)\triangleq \mu(\varrho_{\infty}(1+a))-\mu(\varrho_{\infty}), \ \
\tilde{\lambda}(a)\triangleq\lambda(\varrho_{\infty}(1+a))-\lambda(\varrho_{\infty}), \\
&&\hspace{-6mm}K_{1}(a)\triangleq \frac{\partial_{\varrho}P\left(\varrho_{\infty}(1+a),\mathcal{T}_{\infty}\right)}{(1+a)\partial_{\varrho}P(\varrho_{\infty},\mathcal{T}_{\infty})}-1, \  K_{2}(a)\triangleq  \frac{\chi_{0}}{\varrho_{\infty}}\sqrt{\frac{\mathcal{T}_{\infty}}{C_{v}}}\big(\frac{\pi_{1}\big(\varrho_{\infty}(1+a)\big)}{1+a}-\pi_{1}(\varrho_{\infty})\big),
\\
&&\hspace{-6mm}K_{3}(a)\triangleq \chi_{0} \sqrt{\frac{\mathcal{T}_{\infty}}{C_{v}}}\int_{0}^{a} \frac{\pi'_{1}\left(\varrho_{\infty}(1+z)\right)}{1+z}dz, \ \ \
\tilde{K}_{2}(a)\triangleq\frac{\pi_{1}(\varrho_{\infty}(1+a))}{C_{v}\varrho_{\infty}(1+a)}, \\
&&\hspace{-6mm} \tilde{K}_{1}(a)\triangleq \frac{\chi_{0}}{\varrho_{\infty}} \sqrt{\frac{\mathcal{T}_{\infty}}{C_{v}}}\big(\frac{\pi_{1}\left(\varrho_{\infty}(1+a)\right)}{1+a}-\pi_{1}(\varrho_{\infty})\big), \ \  \tilde{\kappa}(a)\triangleq\kappa(\varrho_{\infty}(1+a))-\kappa(\varrho_{\infty}).
\end{eqnarray*}
Note that $K_{1}$, $K_{2}$, $K_{3}$ $\widetilde{K}_{1}$, $\widetilde{K}_{2}$, $\widetilde{\mu}$, $\widetilde{\lambda}$ and $\widetilde{\kappa}$ are smooth functions satisfyiny $K_{1}(0)=K_{2}(0)=K_{3}(0)=\tilde{K}_{1}(0)=\tilde{\mu}(0)=\tilde{\lambda}(0)=\widetilde{\kappa}(0)=0$.

The main result of the paper is stated as follows.
\begin{thm} \label{Thm1.1} Let $\varrho_{\infty}>0$ and $\mathcal{T}_{\infty}$ be two constant such that \eqref{Eq:1.6} is fulfilled. Suppose that
 $d\geq3$, and that $p$ satisfies
\begin{equation}\label{Eq:1.9}
2\leq p<d \ \ \hbox{and} \ \ p\leq \frac{2d}{d-2}.
\end{equation}
Let $(a,v,\theta)$ be the corresponding global solution to \eqref{Eq:1.8} with the initial data $
(a,v,\theta)|_{t=0}=(a_{0},v_{0},\theta_{0})$, which was constructed in \cite{DH}. Let
\begin{equation}\label{Eq:1.10}
1-\frac{d}{2}<s_{1}\leq s_{0} \Big(s_{0}\triangleq \frac{2d}{p}-\frac{d}{2}\Big).
\end{equation}
There exists a positive constant $c=c\left( p,d,\lambda,\mu,P,\kappa,C_{v},\varrho_{\infty},\mathcal{T}_{\infty}\right)$ such that if
\begin{equation}\label{Eq:1.11}
\mathcal{D}_{p,0}\triangleq \left\|\left(a_{0},v_{0}, \theta_{0}\right)\right\|^{\ell}_{\dot{B}_{2,\infty}^{-s_{1}}}\leq c,
\end{equation}
then it holds that
\begin{equation}\label{Eq:1.12}
\mathcal{D}_{p}(t)\lesssim \big(\mathcal{D}_{p,0}+\left\|\left(\nabla a_{0},v_{0}\right)\right\|^{h}_{\dot{B}_{p,1}^{\frac {d}{p}-1}}+\left\|\theta_{0}\right\|^{h}_{\dot{B}_{p,1}^{\frac {d}{p}-2}}\big) \ \ \hbox{for all} \ \ t\geq 0,
\end{equation}
where the functional $\mathcal{D}_{p}(t)$ is defined by
\begin{eqnarray}\label{Eq:1.13}
\mathcal{D}_{p}(t) &\triangleq &\sup_{s\in[\varepsilon-s_{1},\frac {d}{2}+1]}\|\langle\tau\rangle^{\frac{s_{1}+s}2}(a,v,\theta)\|_{L^{\infty}_{t}(\dot {B}^{s}_{2,1})}^{\ell}+\|\langle\tau\rangle^{\alpha}(\nabla a,v)\|_{\widetilde{L} ^{\infty}_{t}(\dot {B}^{\frac {d}{p}-1}_{p,1})}^{h} \nonumber \\
&&+\|\langle\tau\rangle^{\alpha} \theta \|_{\widetilde{L}^{\infty}_{t}(\dot{B}^{\frac{d}{p}-2}_{p,1})}^{h}+\|\tau^{\alpha}(\nabla v,\theta )\|_{\widetilde{L}^{\infty}_{t}(\dot {B}^{\frac {d}{p}}_{p,1})}^{h}
\end{eqnarray}
with $\alpha\triangleq s_{1}+\frac{d}{2}+\frac{1}{2}-\varepsilon$ for sufficiently small $\varepsilon>0$.
\end{thm}

\begin{rem} \label{Rem1.1}
Theorem \ref{Thm1.1} investigates the case of $s_{1}$ belonging to the whole range $(1-\frac{d}{2}, s_{0}]$, which is open left in \cite{DX2}. The sharp lower bound stems from the elementary time-decay inequality. More precisely,
$$\int^t_{0}\langle t-\tau\rangle^{-\frac{s_1+s}{2}}\langle\tau\rangle^{-\delta}d\tau\lesssim \langle t\rangle^{-\frac{s_1+s}{2}},\ \ \
0\leq\frac{s_1+s}{2}\leq \delta, \ \ \delta>1.
$$
In subsequent low-frequency analysis, the minimum value of $\delta$ is $s_{1}/2+d/4+1/2$, owing to $s\leq d/2+1$. Consequently, $s_{1}/2+d/4+1/2>1$ yields the desired lower bound. In addition, Theorem \ref{Thm1.1} holds in case that $\mu$, $\lambda$ and $\kappa$ depend smoothly on the density.
\end{rem}

\begin{rem} \label{Rem1.2}
If \eqref{Eq:1.11} is replaced by the slightly stronger hypothesis:
\begin{equation*}
\|(a_{0},v_{0}, \theta_{0})\|^{\ell}_{\dot{B}_{2,1}^{-s_{1}}}\leq  c \ll 1,
\end{equation*}
then one can take $\varepsilon=0$ in both $\alpha$ and $\mathcal{D}_{p}(t)$.
\end{rem}

As a consequence of Theorem \ref{Thm1.1}, the time decay estimates of the $L^{p}$ norm (the slightly stronger $\dot{B}^{0}_{p,1}$ norm in fact) of solutions.
\begin{cor} \label{Cor1.1}
Under the additional assumption \eqref{Eq:1.10}-\eqref{Eq:1.11}, the global solution satisfies
\begin{eqnarray*}
	\|\Lambda^{s}a\|_{L^{r}}\lesssim A_{0}
\langle t\rangle^{-\frac{s_{1}+s}{2}-\frac{d}{2}\big(\frac{1}{2}-\frac{1}{r}\big)}\ \ \hbox{if} \  -\tilde{s}_{1}<s+d\left(\frac{1}{p}-\frac{1}{r}\right)\leq\frac{d}{p},\\
\|\Lambda^{s}v\|_{L^{r}}\lesssim A_{0} t^{-\frac{s_{1}+s}{2}-\frac{d}{2}\big(\frac{1}{2}-\frac{1}{r}\big)}\ \
\hbox{if} \ -\tilde{s}_{1}<s+d\left(\frac{1}{p}-\frac{1}{r}\right)\leq\frac{d}{p}+1,\\
\|\Lambda^{s}\theta\|_{L^{r}}\lesssim A_{0} t^{-\frac{s_{1}+s}{2}-\frac{d}{2}\big(\frac{1}{2}-\frac{1}{r}\big)}\ \
\hbox{if} \ -\tilde{s}_{1}<s+d\left(\frac{1}{p}-\frac{1}{r}\right)\leq\frac{d}{p},
	\end{eqnarray*}
for $t\geq 0$ and $p\leq r\leq\infty$, where $\tilde{s}_{1}\triangleq s_{1}+d\big(\frac{1}{2}-\frac{1}{p}\big),  A_{0}\triangleq\big(\mathcal{D}_{p,0}+\|(\nabla a_{0},v_{0})\|^{h}_{\dot{B}_{p,1}^{\frac {d}{p}-1}}+\|\theta_{0}\|^{h}_{\dot{B}_{p,1}^{\frac {d}{p}-2}}\big)$ and the operator $\Lambda^{s}$ is defined by $\Lambda^{s}f\triangleq\mathcal{F}^{-1}\left(|\xi|^{s}\mathcal{F}f\right)$ for $s\in\mathbb{R}$.
\end{cor}

\begin{rem} \label{Rem1.3}
For convenience of reader, let us show the decay rates of heat kernel $E(t)U_{0}\triangleq e^{-t\Delta}U_{0}$ first.
In Fourier variable, we have
$$\mathcal{F}[E(t)U_{0}](\xi)=e^{-t|\xi|^2}\mathcal{F}U_{0}(\xi).$$
It follows from  Hausdorff-Young and H\"{o}lder inequalities that
\begin{eqnarray*}
\|E(t)U_{0}\|_{L^p}\leq \|\mathcal{F}[E(t)U_{0}](\xi)\|_{L^{p'}}\leq\|\mathcal{F}U_{0}(\xi)\|_{L^{q'}}\|e^{-t|\xi|^2}\|_{L^{m}}\lesssim \|U_0\|_{L^q}t^{-\frac{d}{2m}},
\end{eqnarray*}
where $1/p+1/p'=1/q+1/q'=1$, $1/p'=1/q'+1/m$ and $p\geq2$. Hence, one can get $m=p$ if choosing $q=p/2$, that is, the heat kernel enjoys the time-decay rate of
$O(t^{-\frac{d}{2p}})$ in $L^{p}$ norm if $U_{0}\in L^{p/2}$. Noticing the embedding $L^{p/2}\hookrightarrow\dot{B}_{2,\infty}^{-s_{0}}$, we see that the global solution of \eqref{Eq:1.8} decays to constant equilibrium with the same rate if taking the endpoint regularity $s_{1}=s_{0}$. Those decay rates in Corollary \ref{Cor1.1} are thus optimal and satisfactory.
\end{rem}

Prompted by the recent work dedicated to the compressible barotropic flow (see \cite{DX1}), we here aim at proving Theorem \ref{Thm1.1}. The additional unknown $\theta$ \textit{cannot contribute more regularities} in term of \eqref{scaling}, so those nonlinear terms between density, velocity field and temperature need to be treated carefully. Up to now, the global-in-time existence and large-time behavior of solutions of \eqref{Eq:1.1} \textit{remains open in dimension two},
which is left for future consideration. In contrast to \cite{DX2}, the low-frequency analysis for $1-d/2<s_{1}<\max(0,2-d/2)$ is much more technical. Owing to the heat smoothing effect, it is possible to adapt the standard Duhamel principle treating the nonlinear right-hand side $(f,g,k)$ of \eqref{Eq:1.8}.
Precisely, we split the nonlinear term $(f,g,k)$ into $(f^{\ell},g^{\ell},k^{\ell})$ and $(f^{h},g^{h},k^{h})$ (see the context below). In order to handle $(f^{\ell},g^{\ell},k^{\ell})$ in the time-weighted integral, some new and non-standard Besov product estimates are well developed, see \eqref{Eq:3.6}-\eqref{Eq:3.8}.
Secondly, bounding the term $(f^{h},g^{h},k^{h})$ for example, $k_1(a,\theta^h)$ is more elaborate due to the less regularity of $\theta$, where different Sobolev embeddings are mainly employed. See Lemmas \ref{Lem3.2}-\ref{Lem3.3} for more details.

On the other hand, we proceed differently for the analysis of the the high frequencies decay of the solution, since there is no smooth effect for $a$. Indeed,
the idea is to work with a so-called ``\textit{effective velocity}" $w$ (which was initiated by Hoff \cite{HD} and first used in the context of critical regularity by Haspot \cite{HB}) such that, up to low order terms, the divergence-free part of $v$, the temperature $\theta$ and $w$ satisfy a parabolic system while $a$ fulfills a damped transport equation. Then, by employing $L^p$ energy argument directly on these equations after localization, one can eventually obtain optimal decay exponents for high frequencies.

The rest of the paper unfolds as follows. In Section 2, we recall Littlewood-Paley decomposition, Besov spaces and related analysis tools. Section 3 is devoted to the proofs of Theorem \ref{Thm1.1} and Corollary \ref{Cor1.1}.
%********************************************
\section{Preliminary}\setcounter{equation}{0}
Throughout the paper, $C>0$ stands for a generic ``constant''. For brevity, we write
$f\lesssim g$ instead of $f\leq Cg$. The notation $f\approx g$ means that $%
f\lesssim g$ and $g\lesssim f$. For any Banach space $X$ and $f,g\in X$, we agree that $\|(f,g)\|_{X}\triangleq \|f\| _{X}+\|g\|_{X}$. For all $T>0$ and $\rho \in[1,+\infty]$, we denote by $L_{T}^{\rho}(X) \triangleq L^{\rho}([0,T];X)$ the set of measurable functions $f:[0,T]\rightarrow X$ such that $t\mapsto\|f(t)\|_{X}$ is in $L^{\rho}(0,T)$.

Let us next briefly recall Littlewood-Paley decomposition, Besov spaces and analysis tools. The interested reader is referred to Chap. 2 and Chap. 3 of \cite{BCD} for more details. We begin with the homogeneous Littlewood-Paley decomposition. To this end,  we fix some smooth
radial non increasing function $\chi $ with $\mathrm{Supp}\,\chi \subset
B\left(0,\frac {4}{3}\right)$ and $\chi \equiv 1$ on $B\left(0,\frac
{3}{4}\right)$, then set $\varphi (\xi) =\chi (\xi/2)-\chi (\xi)$
so that
$$\sum_{j\in \mathbb{Z}}\varphi ( 2^{-j}\cdot ) =1\ \ \hbox{in}\ \
\mathbb{R}^{d}\setminus \{ 0\} \ \ \hbox{and}\ \ \mathrm{Supp}\,\varphi \subset \left\{ \xi \in \mathbb{R}^{d}:3/4\leq |\xi|\leq 8/3\right\} .$$
The homogeneous dyadic blocks $\dot{\Delta}_{j}$ are defined by
$$\dot{\Delta}_{j}f\triangleq \varphi (2^{-j}D)f=\mathcal{F}^{-1}\left(\varphi
(2^{-j}\cdot )\mathcal{F}f\right)=2^{jd}h(2^{j}\cdot )\star f\ \ \hbox{with}\ \
h\triangleq \mathcal{F}^{-1}\varphi.$$
Formally, we have the homogeneous decomposition as follows
\begin{equation} \label{Eq:2.1}
f=\sum_{j\in \mathbb{Z}}\dot{\Delta}_{j}f,
\end{equation}
for any tempered distribution $f\in S'(\mathbb{R}^{d})$. As it holds only modulo polynomials,
it is convenient to consider the subspace of those tempered distributions $f$ such that
\begin{equation}\label{Eq:2.2}
\lim_{j\rightarrow -\infty }\| \dot{S}_{j}f\| _{L^{\infty} }=0,
\end{equation}
where $\dot{S}_{j}f$ stands for the low frequency cut-off $\dot{S}_{j}f\triangleq\chi (2^{-j}D)f$. As a matter of fact,
if \eqref{Eq:2.2} is fulfilled, then \eqref{Eq:2.1} holds in $S'(\mathbb{R}^{d})$. For convenience, we denote by $S'_{0}(\mathbb{R}^{d})$ the subspace of tempered distributions satisfying \eqref{Eq:2.2}.

With the aid of the Littlewood-Paley decomposition, the homogeneous Besov space is defined as follows.
\begin{defn}\label{Defn2.1}
	For $\sigma\in \mathbb{R}$ and $1\leq p,r\leq\infty,$ the homogeneous
	Besov spaces $\dot{B}^{\sigma}_{p,r}$ is defined by
	$$\dot{B}^{\sigma}_{p,r}\triangleq\left\{f\in S'_{0}:\|f\|_{\dot{B}^{\sigma}_{p,r}}<+\infty\right\},$$
	where
	\begin{equation}\label{Eq:2.3}
	\|f\|_{\dot B^{\sigma}_{p,r}}\triangleq \|(2^{j\sigma}\|\dot{\Delta}_{j}  f\|_{L^p})\|_{\ell^{r}(\mathbb{Z})}.
	\end{equation}
\end{defn}
In many parts of this paper, we use the following classical properties (see \cite{BCD}):

$\bullet$ \ \emph{Scaling invariance:} For any $\sigma\in \mathbb{R}$ and $(p,r)\in
[1,\infty ]^{2}$, there exists a constant $C=C(\sigma,p,r,d)$ such that for all $\lambda >0$ and $f\in \dot{B}_{p,r}^{\sigma}$, we have
$$
C^{-1}\lambda ^{\sigma-\frac {d}{p}}\|f\|_{\dot{B}_{p,r}^{\sigma}}
\leq \|f(\lambda \cdot)\|_{\dot{B}_{p,r}^{\sigma}}\leq C\lambda ^{\sigma-\frac {d}{p}}\|f\|_{\dot{B}_{p,r}^{\sigma}}.
$$

$\bullet$ \ \emph{Completeness:} $\dot{B}^{\sigma}_{p,r}$ is a Banach space whenever $%
\sigma<\frac{d}{p}$ or $\sigma\leq \frac{d}{p}$ and $r=1$.

$\bullet$ \ \emph{Interpolation:} The following inequality is satisfied for $1\leq p,r_{1},r_{2}, r\leq \infty, \sigma_{1}\neq \sigma_{2}$ and $\theta_{1} \in (0,1)$:
$$\|f\|_{\dot{B}_{p,r}^{\theta_{1} \sigma_{1}+(1-\theta_{1})\sigma_{2}}}\lesssim \|f\| _{\dot{B}_{p,r_{1}}^{\sigma_{1}}}^{\theta_{1}} \|f\|_{\dot{B}_{p,r_2}^{\sigma_{2}}}^{1-\theta_{1} }$$
with $\frac{1}{r}=\frac{\theta_{1}}{r_{1}}+\frac{1-\theta_{1}}{r_{2}}$.

$\bullet$ \ \emph{Action of Fourier multipliers:} If $F$ is a smooth homogeneous of
degree $m$ function on $\mathbb{R}^{d}\backslash \{0\}$ then
$$F(D):\dot{B}_{p,r}^{\sigma}\rightarrow \dot{B}_{p,r}^{\sigma-m}.$$

The following embedding properties are used several times in this paper.
\begin{prop} \label{Prop2.1} (Embedding for Besov spaces on $\mathbb{R}^{d}$)
	\begin{itemize}
		\item For any $p\in[1,\infty]$ we have the  continuous embedding
		$\dot {B}^{0}_{p,1}\hookrightarrow L^{p}\hookrightarrow \dot {B}^{0}_{p,\infty}.$
		\item If $\sigma\in\mathbb{R}$, $1\leq p_{1}\leq p_{2}\leq\infty$ and $1\leq r_{1}\leq r_{2}\leq\infty,$
		then $\dot {B}^{\sigma}_{p_1,r_1}\hookrightarrow
		\dot {B}^{\sigma-d\,(\frac{1}{p_{1}}-\frac{1}{p_{2}})}_{p_{2},r_{2}}$.
		\item The space  $\dot {B}^{\frac {d}{p}}_{p,1}$ is continuously embedded in the set  of
		bounded  continuous functions (going to zero at infinity if, additionally, $p<\infty$).
	\end{itemize}
\end{prop}
Let us mention the following product estimate in the Besov spaces, which plays a fundamental role in bounding bilinear terms of \eqref{Eq:1.8} (see \cite{BCD,DX1}).
\begin{prop}\label{Prop2.2}
	Let $\sigma>0$ and $1\leq p,r\leq\infty$. Then $\dot{B}^{\sigma}_{p,r}\cap L^{\infty}$ is an algebra and
	$$
	\|fg\|_{\dot{B}^{\sigma}_{p,r}}\lesssim \left\|f\right\|_{L^{\infty}}\|g\|_{\dot{B}^{\sigma}_{p,r}}+\|g\|_{L^{\infty}}\|f\|_{\dot{B}^{\sigma}_{p,r}}.
	$$
	Let the real numbers $\sigma_{1},$ $\sigma_{2},$ $p_{1}$  and $p_{2}$ fulfill
	$$
	\sigma_{1}+\sigma_{2}>0,\quad \sigma_{1}\leq\frac {d}{p_{1}},\quad\sigma_{2}\leq\frac {d}{p_{2}},\quad
	\sigma_{1}\geq\sigma_{2},\quad\frac{1}{p_{1}}+\frac{1}{p_{2}}\leq1.
	$$
	Then we have
	$$\|fg\|_{\dot{B}^{\sigma_{2}}_{q,1}}\lesssim \|f\|_{\dot{B}^{\sigma_{1}}_{p_{1},1}}\|g\|_{\dot{B}^{\sigma_{2}}_{p_{2},1}}\quad\hbox{with}\quad
	\frac1{q}=\frac1{p_{1}}+\frac1{p_{2}}-\frac{\sigma_{1}}{d}.$$
	Additionally, for exponents $\sigma>0$ and $1\leq p_{1},p_{2},q\leq\infty$ satisfying
	$$\frac{d}{p_{1}}+\frac{d}{p_{2}}-d\leq \sigma \leq\min\left(\frac {d}{p_{1}},\frac {d}{p_{2}}\right)\quad\hbox{and}\quad \frac{1}{q}=\frac {1}{p_{1}}+\frac {1}{p_{2}}-\frac{\sigma}{d},$$
	we have
	$$\|fg\|_{\dot{B}^{-\sigma}_{q,\infty}}\lesssim\|f\|_{\dot{B}^{\sigma}_{p_{1},1}}\|g\|_{\dot{B}^{-\sigma}_{p_{2},\infty}}.$$
\end{prop}
Proposition \ref{Prop2.2} are not enough to bound all nonlinear terms in the proof of Theorem \ref{Thm1.1}, so we need to the following non standard product estimates (see \cite{DX1,XX}).
\begin{prop}\label{Prop2.3}
	Let the real numbers $\sigma_{1},$ $\sigma_{2},$ $p_{1}$  and $p_{2}$ be such that
	$$
	\sigma_{1}+\sigma_{2}\geq0,\quad \sigma_{1}\leq\frac {d}{p_{1}},\quad\sigma_{2}< \min\left(\frac {d}{p_{1}},\frac {d}{p_{2}}\right)\quad \hbox{and} \quad\frac{1}{p_{1}}+\frac{1}{p_{2}}\leq1.
	$$
	Then it holds that
	$$\|fg\|_{\dot{B}^{\sigma_{1}+\sigma_{2}-\frac{d}{p_{1}}}_{p_{2},\infty}}\lesssim \|f\|_{\dot{B}^{\sigma_{1}}_{p_{1},1}}\left\|g\right\|_{\dot{B}^{\sigma_{2}}_{p_{2},\infty}}.$$
\end{prop}
\begin{prop}\label{Prop2.4}
	Let $j_{0}\in\mathbb{Z}$, and denote $z^{\ell}\triangleq\dot S_{j_{0}}z$, $z^{h}\triangleq z-z^{\ell}$ and, for any $\sigma\in\mathbb{R}$,
	$$
	\|z\|_{\dot B^{\sigma}_{2,\infty}}^{\ell}\triangleq\sup_{j\leq j_{0}}2^{j\sigma}\|\dot{\Delta}_{j} z\|_{L^{2}}.
	$$
	There exists a universal integer $N_{0}$ such that  for any $2\leq p\leq 4$ and $\sigma>0,$ we have
	\begin{eqnarray}\label{Eq:2.4}
	&&\|fg^{h}\|_{\dot {B}^{-s_{0}}_{2,\infty}}^{\ell}\leq C \left(\|f\|_{\dot {B}^{\sigma}_{p,1}}+\|\dot {S}_{j_{0}+N_{0}}f\|_{L^{{p}^{*}}}\right)\|g^{h}\|_{\dot{B}^{-\sigma}_{p,\infty}},\\
	\label{Eq:2.5}
	&&\|f^{h}g\|_{\dot {B}^{-s_{0}}_{2,\infty}}^{\ell}
	\leq C \left(\|f^{h}\|_{\dot{B}^{\sigma}_{p,1}}+\|\dot{S}_{j_{0}+N_{0}}f^{h}\|_{L^{p^{*}}}\right)\|g\|_{\dot {B}^{-\sigma}_{p,\infty}}
	\end{eqnarray}
	with $s_{0}\triangleq \frac{2d}{p}-\frac {d}{2}$ and $\frac1{p^{*}}\triangleq\frac{1}{2}-\frac{1}{p}$,
	and $C$ depending only on $j_{0}$, $d$ and $\sigma$.
\end{prop}
System \eqref{Eq:1.8} also involves compositions of functions (through $K_{1}(a)$, $K_{2}(a)$, $K_{3}(a)$ $\tilde{K}_{1}(a)$, $\tilde{K}_{2}(a)$, $\tilde{\mu}(a)$, $\tilde{\lambda}(a)$ and $\widetilde{\kappa}(a)$) that
are handled due to the following proposition.
\begin{prop}\label{Prop2.5}
	Let $F:\mathbb{R}\rightarrow\mathbb{R}$ be  smooth with $F(0)=0$.
	For  all  $1\leq p,r\leq\infty$ and $\sigma>0$, we have
	$F(f)\in \dot {B}^{\sigma}_{p,r}\cap L^{\infty}$  for  $f\in \dot {B}^{\sigma}_{p,r}\cap L^{\infty},$  and
	$$\|F(f)\|_{\dot {B}^{\sigma}_{p,r}}\leq C\|f\|_{\dot{B}^{\sigma}_{p,r}}$$
	with $C$ depending only on $\|f\|_{L^{\infty}}$, $F'$ (and higher derivatives), $\sigma$, $p$ and $d$.
	
	In the case $\sigma>-\min\left(\frac {d}{p},\frac {d}{p'}\right)$ then $f\in\dot {B}^{\sigma}_{p,r}\cap\dot {B}^{\frac {d}{p}}_{p,1}$
	implies that $F(f)\in \dot{B}^{\sigma}_{p,r}\cap \dot {B}^{\frac{d}{p}}_{p,1}$, and we have
	$$\|F(f)\|_{\dot B^{\sigma}_{p,r}}\leq C(1+\|f\|_{\dot {B}^{\frac {d}{p}}_{p,1}})\|f\|_{\dot {B}^{\sigma}_{p,r}}.$$
\end{prop}
In addition, we also notice the classical \emph{Bernstein inequality}:
\begin{equation}\label{Eq:2.6}
\|D^{k}f\|_{L^{b}}
\leq C^{1+k} \lambda^{k+d\,(\frac{1}{a}-\frac{1}{b})}\|f\|_{L^{a}}
\end{equation}
that holds for all function $f$ such that $\mathrm{Supp}\,\mathcal{F}f\subset\left\{\xi\in \mathbb{R}^{d}: |\xi|\leq R\lambda\right\}$ for some $R>0$
and $\lambda>0$, if $k\in\mathbb{N}$ and $1\leq a\leq b\leq\infty$.

More generally, if we suppose $f$ to satisfy $\mathrm{Supp}\,\mathcal{F}f\subset \left\{\xi\in \mathbb{R}^{d}:
R_{1}\lambda\leq\left|\xi\right|\leq R_{2}\lambda\right\}$ for some $0<R_{1}<R_{2}$  and $\lambda>0$,
then for any smooth  homogeneous of degree $m$ function $A$ on $\mathbb{R}^d\setminus\{0\}$ and $1\leq a\leq\infty,$ we get
(see e.g. Lemma 2.2 in \cite{BCD}):
\begin{equation}\label{Eq:2.7}
\|A(D)f\|_{L^{a}}\approx\lambda^{m}\|f\|_{L^{a}}.
\end{equation}
An obvious consequence of \eqref{Eq:2.6} and \eqref{Eq:2.7} is that
$\|D^{k}f\|_{\dot{B}^{s}_{p, r}}\thickapprox \|f\|_{\dot{B}^{s+k}_{p, r}}$ for all $k\in\mathbb{R}$.

%Let us recall those inequalities of Gagliardo-Nirenberg type,
%which can be found by the work of Sohinger and Strain \cite{SS} (see also \cite{BCD,DX1,SX,XX}).
%\begin{prop}\label{Prop2.6}
	%The following interpolation inequalities hold:
	%\begin{equation*}
	%\|\Lambda ^{\beta}f\|_{L^{r}}\lesssim \|\Lambda ^{\beta_{1}}f\|^{1-\theta_{1}}_{L^{q}} \|\Lambda ^{\beta_{2}}f\|^{\theta_{1}}_{L^{q}},
	%\end{equation*}
	%whenever $0\leq\theta_{1}\leq1$, $1\leq q\leq r\leq\infty$ and
	%\begin{equation*}
	%\beta+d\left(\frac{1}{q}-\frac{1}{r}\right)=\beta_{1}(1-\theta_{1})+\beta_{2}\theta_{1}.
	%\end{equation*}
%\end{prop}
In order to state optimal regularity estimates for the heat equation, a class of mixed space-time Besov spaces are also used, which was initiated by J.-Y. Chemin and N. Lerner \cite{CL} (see also \cite{CJY} for the particular case of Sobolev spaces).
\begin{defn}\label{Defn2.2}
	For $T>0, \sigma\in\mathbb{R}, 1\leq r,\rho\leq\infty$, the homogeneous Chemin-Lerner space $\tilde{L}^{\rho}_{T}(\dot{B}^{\sigma}_{p,r})$
	is defined by
	$$\tilde{L}^{\rho}_{T}(\dot{B}^{\sigma}_{p,r})\triangleq\left\{f\in L^{\rho}\left(0,T;S'_{0}\right):\|f\|_{\tilde{L}^{\rho}_{T}(\dot{B}^{\sigma}_{p,r})}<+\infty\right\},$$
	where
	\begin{equation}\label{Eq:2.8}
	\|f\|_{\tilde{L}^{\rho}_{T}(\dot{B}^{\sigma}_{p,r})}\triangleq\|(2^{j\sigma}\|\dot{\Delta}_{j} f\|_{L^{\rho}_{T}(L^{p})})\|_{\ell^{r}(\mathbb{Z})}.
	\end{equation}
\end{defn}
For notational simplicity, index $T$ is omitted if $T=+\infty $. We denote
\begin{equation*}
\tilde{\mathcal{C}}_{b}(\mathbb{R_{+}};\dot{B}_{p,r}^{\sigma})\triangleq \left\{f \in
\mathcal{C}(\mathbb{R_{+}};\dot{B}_{p,r}^{\sigma})\ \hbox{s.t}\ \|f\| _{\tilde{L}^{\infty}(\dot{B}_{p,r}^{\sigma})}<+\infty \right\} .
\end{equation*}
Furthermore, Minkowski¡¯s inequality
allows us to compare $\|\cdot\|_{\tilde{L}^{\rho}_{T}(\dot{B}^{\sigma}_{p,r})}$ with the more standard Lebesgue-Besov semi-norms of $L_{T}^{\rho} (\dot{B}_{p,r}^{\sigma})$ as follows.
\begin{rem}\label{Rem2.1}
	It holds that
	$$\|f\|_{\tilde{L}^{\rho}_{T}(\dot{B}^{\sigma}_{p,r})}\leq\|f\|_{L^{\rho}_{T}(\dot{B}^{\sigma}_{p,r})}\,\,\,
	\mbox{if} \,\, \, r\geq\rho;\ \ \ \
	\|f\|_{\tilde{L}^{\rho}_{T}(\dot{B}^{\sigma}_{p,r})}\geq\|f\|_{L^{\rho}_{T}(\dot{B}^{\sigma}_{p,r})}\,\,\,
	\mbox{if}\,\,\, r\leq\rho.$$
\end{rem}
Restricting the above norms \eqref{Eq:2.3} and \eqref{Eq:2.8} to the low or high
frequencies parts of distributions will be fundamental in our method. For that pourpose, we shall often use the following notation for some suitable integer $j_{0}$ \footnote{Note that for technical reasons, we need a small overlap between low and high frequencies.}
\begin{equation*}
\| f\| _{\dot{B}_{p,1}^{\sigma}}^{\ell} \triangleq \sum_{j\leq
	j_{0}}2^{j\sigma}\| \dot{\Delta}_{j}f\|_{L^{p}} \ \mbox{and} \ \|f\|_{\dot{B}_{p,1}^{\sigma}}^{h}\triangleq \sum_{j\geq j_{0}-1}2^{j\sigma}\| \dot{\Delta}_{j}f\| _{L^{p}},
\end{equation*}
\begin{equation*}
\|f\| _{\tilde{L}_{T}^{\infty} (\dot{B}_{p,1}^{\sigma})}^{\ell} \triangleq
\sum_{j\leq j_{0}}2^{j\sigma}\|\dot{\Delta}_{j}f\|_{L_{T}^{\infty} (L^{p})} \
\mbox{and} \ \|f\| _{\tilde{L}_{T}^{\infty} (\dot{B}_{p,1}^{\sigma})}^{h}\triangleq \sum_{j\geq j_{0}-1}2^{j\sigma}\| \dot{\Delta}_{j}f\|
_{L_{T}^{\infty} (L^{p})}.
\end{equation*}

Finally, we end this section with the parabolic regularity estimates for the
heat equation.
\begin{prop}\label{Prop2.7}
Let $\sigma\in \mathbb{R}$, $(p,r)\in [1,\infty]^{2}$ and $1\leq \rho_{2}\leq\rho_{1}\leq\infty$. Let $u$ satisfy
\begin{equation*}	
\left\{\begin{array}{lll}
	\d_{t}u-\mu\Delta u=f,\\
	u_{|t=0}=u_{0}.
	\end{array}
	\right.
\end{equation*}
Then for all $T>0$ the following a priori estimate is fulfilled:
\begin{equation}\label{Eq:2.9}
	\mu^{\frac{1}{\rho_{1}}}\|u\|_{\tilde L_{T}^{\rho_{1}}(\dot B^{\sigma+\frac{2}{\rho_1}}_{p,r})}\lesssim
	\|u_{0}\|_{\dot{B}^{\sigma}_{p,r}}+\mu^{\frac{1}{\rho_{2}}-1}\|f\|_{\tilde L^{\rho_{2}}_{T}(\dot {B}^{\sigma-2+\frac{2}{\rho_{2}}}_{p,r})}.
\end{equation}
\end{prop}
\begin{rem} \label{Rem2.2}
	The solutions to the following \emph{Lam\'e system}
	\begin{equation}\label{Eq:2.10}
	\left\{\begin{array}{lll}\d_{t}u-\mu\Delta u-\left(\lambda+\mu\right)\nabla \mathrm{div}\,u=f,\\
	u_{|t=0}=u_{0},
	\end{array}
	\right.
	\end{equation}
	where $\lambda$ and $\mu$ are constant coefficients such that $\mu>0$ and $\lambda+2\mu>0,$
	also fulfill \eqref{Eq:2.9} (up to the dependence w.r.t. the viscosity). Indeed, if we denote by $\mathcal{P}\triangleq\mathrm{Id}-\nabla(-\Delta)^{-1}\mathrm{div}$
	and $\mathcal{Q}\triangleq\mathrm{Id}-\mathcal{P}$ the orthogonal projectors over divergence-free
	and potential vector fields, then we see both $\mathcal{P}u$ and $\mathcal{Q} u$ satisfy the heat equation, as it can easily be observed by applying $\mathcal{P}$ and $\mathcal{Q}$ to \eqref{Eq:2.10}.
\end{rem}

%************************************

\section{The proof of Time-decay estimates} \setcounter{equation}{0}
This section is devoted to the proof of Theorem \ref{Thm1.1} taking for granted the global existence result in \cite{DH}. We denote by $\mathcal{X}_{p}(t)$ the energy norm:
\begin{eqnarray}\label{Eq:3.1}
	\mathcal{X}_{p}(t)&\triangleq&\|(a,v,\theta)\|_{\widetilde{L}_{t}^{\infty} (\dot{B}_{2,1}^{\frac {d}{2}-1})}^{\ell}
	+\|(a,v,\theta)\|_{L_{t}^{1} (\dot{B}_{2,1}^{\frac{d}{2}+1})}^{\ell} +\|(\nabla a,v)\|_{\widetilde{L}_{t}^{\infty}(\dot{B}_{p,1}^{\frac {d}{p}-1})}^{h}\nonumber \\ [1mm]
	&&+\|\theta\|_{\widetilde{L}_{t}^{\infty}(\dot{B}_{p,1}^{\frac {d}{p}-2})}^{h}+\|(a,\nabla v,\theta)\|_{L_{t}^{1}(\dot{B}_{p,1}^{\frac {d}{p}})}^{h}.
\end{eqnarray}
In what follows, we shall use repeatedly the following obvious inequality that is satisfied whenever $0\leq \sigma_{1}\leq \sigma_{2}$ and $\sigma_{2}>1$:
\begin{equation}\label{Eq:3.2}
\int_{0}^{t}\langle t-\tau\rangle^{-\sigma_{1}}\langle\tau\rangle^{-\sigma_{2}}d\tau\lesssim\langle t\rangle^{-\sigma_{1}}.
\end{equation}
Let us keep in mind that the global solution $(a,v,\theta)$ satisfies
\begin{equation}\label{Eq:3.3}
\left\|a\right\|_{\tilde{L}^{\infty}_{t}(\dot{B}^{\frac{d}{p}}_{p,1})}\leq c \ll 1 \ \ \hbox{for all} \ \ t\geq0.
\end{equation}
\subsection{First step: Bounds for the low frequencies}
Let $(E(t))_{t\geq0}$ be the semi-group associated with the left-hand side of \eqref{Eq:1.8}. The standard Duhamel principle yields
\begin{equation*}
\left(
\begin{array}{c}
a(t)\\
v(t)\\
\theta(t)
\end{array}
\right)
=E(t)
\left(
\begin{array}{c}
a_{0}\\
v_{0}\\
\theta_{0}
\end{array}
\right)
+\int_{0}^{t}E(t-\tau)
\left(
\begin{array}{c}
f(\tau)\\
g(\tau)\\
k(\tau)
\end{array}
\right)d\tau.
\end{equation*}
First of all, we state smoothing estimate of the linearized solution $\left(a_{L},v_{L},\theta_{L}\right)\triangleq E(t)\left(a_{0},v_{0},\theta_{0}\right)$, which behaves like that of heat kernel.
\begin{lem}\label{Lem3.1}
	Let $\left(a_{L},v_{L},\theta_{L}\right)$ be the solution to the following system
\begin{equation*}
	\left\{
	\begin{array}{l}
	\partial _{t}a_{L}+\mathrm{div}\,v_{L}=0, \\[2mm]
	\partial _{t}v_{L}-\tilde{\mathcal{A}}v_{L}+\nabla a_{L}+\gamma \nabla \theta_{L}=0,\\[2mm]
	\partial_{t}\theta_{L}-\beta\Delta\theta_{L}+\gamma\,\mathrm{div}\,v_{L}=0
	\end{array}
	\right.
\end{equation*}
with the initial data
\begin{equation*}
	(a_{L},v_{L},\theta_{L})|_{t=0}=(a_{0},v_{0},\theta_{0}).
\end{equation*}
Then, for any $j_{0}\in \mathbb{Z}$, there exists a positive constant $c_{0}=c_{0}\left(\lambda_{\infty},\mu_{\infty},\beta,\gamma,j_{0}\right)$ such that
\begin{equation*}
	\|(a_{L,j},v_{L,j},\theta_{L,j})(t)\|_{L^{2}}\lesssim e^{-c_{0}2^{2j}t} \|(a_{0,j},v_{0,j},\theta_{0,j})\|_{L^{2}}
\end{equation*}
for $t\geq0$ and $j\leq j_{0}$, where we set $z_{j}=\dot{\Delta}_{j}z$ for any $z\in S'(\mathbb{R}^{d})$.
\end{lem}
The interested reader is referred to \cite{DX2} for the proof of Lemma \ref{Lem3.1}.
Set $U\triangleq (a,v,\theta)$ and $U_{0}\triangleq (a_{0},v_{0},\theta_{0})$. From  Lemma \ref{Lem3.1}, we perform the same procedure as in \cite{DX1,DX2} to obtain for $s+s_{1}>0$,
\begin{equation*}
\sup_{t\geq 0}t^{\frac{s_{1}+s}{2}}\|E(t)U_{0}\|^{\ell}_{\dot{B}^{s}_{2,1}}
\lesssim \|U_{0}\|^{\ell}_{\dot{B}^{-s_{1}}_{2,\infty}}.
\end{equation*}
Additionally, it is clear that for $s+s_{1}>0$,
\begin{equation*}
\|E(t)U_{0}\|^{\ell}_{\dot{B}^{s}_{2,1}}\lesssim \|U_{0}\|^{\ell}_{\dot{B}^{-s_{1}}_{2,\infty}}\sum_{j\leq j_{0}}2^{j\left(s_{1}+s\right)}\lesssim \|U_{0}\|^{\ell}_{\dot{B}^{-s_{1}}_{2,\infty}}.
\end{equation*}
Then it follows that
\begin{equation*}
\sup_{t\geq 0}\langle t\rangle^{\frac{s_{1}+s}{2}}\|E(t)U_{0}\|^{\ell}_{\dot{B}^{s}_{2,1}}\lesssim \|U_{0}\|^{\ell}_{\dot{B}^{-s_{1}}_{2,\infty}}
\ \ \hbox{with} \ \ \langle t\rangle\triangleq\sqrt{1+t^{2}}.
\end{equation*}
Consequently, with the aid of Duhamel formula, we end up with
\begin{equation}\label{Eq:3.4}
\|(a,v,\theta)\|^{\ell}_{\dot{B}^{s}_{2,1}}\lesssim \langle \tau\rangle^{-\frac{s_{1}+s}{2}}\|(a_{0},v_{0},\theta_{0})\|^{\ell}_{\dot{B}^{-s_{1}}_{2,\infty}}+\int_{0}^{t}
\langle t-\tau\rangle^{-\frac{s_{1}+s}{2}}\|(f,g,k)(\tau)\|^{\ell}_{\dot{B}^{-s_{1}}_{2,\infty}}d\tau.
\end{equation}

Bounding the time-weighted integral on the right side of \eqref{Eq:3.4} is included in the following proposition.
\begin{prop}\label{Prop3.1}
Let $p$ fulfills \eqref{Eq:1.9}, then it holds that for all $t\geq 0$,
\begin{equation*}
\int_{0}^{t}\langle t-\tau\rangle^{-\frac{s_{1}+s}{2}} \|(f,g,k)(\tau)\|^{\ell}_{\dot{B}^{-s_{1}}_{2,\infty}}d\tau\lesssim \langle t\rangle^{-\frac{s_{1}+s}{2}}\left(\mathcal{D}^{2}_{p}(t)+\mathcal{X}^{2}_{p}(t)\right),
\end{equation*}
provided that $-s_{1}<s\leq \frac{d}{2}+1$, where $\mathcal{X}_{p}(t)$ and $\mathcal{D}_{p}(t)$ have been defined in \eqref{Eq:3.1} and \eqref{Eq:1.13}, respectively.
\end{prop}
Indeed, we decompose the nonlinear term $g=\sum \limits_{i=1}^{6}\mathcal{G}_{i}$ with
\begin{eqnarray*}
&&\mathcal{G}_{1}=-v\cdot \nabla v, \ \ \ \ \mathcal{G}_{2}=-K_{1}(a)\nabla a, \\ &&\mathcal{G}_{3}=\frac{1}{(1+a)\nu}\left(2\widetilde{\mu}(a)\mathrm{div} D(u)+\widetilde{\lambda}(a)\nabla \mathrm{div}\,v\right)-I(a)\widetilde{\mathcal{A}}v,\\
&&\mathcal{G}_{4}=\frac{1}{(1+a)\nu}\left(\widetilde{\mu}'(a) D(u)\cdot \nabla a
+\widetilde{\lambda}'(a) \mathrm{div}\,v \nabla a\right),\\
&& \mathcal{G}_{5}=-K_{2}(a)\nabla \theta, \ \ \ \ \ \ \mathcal{G}_{6}=-\theta \nabla K_{3}(a).
\end{eqnarray*}
As shown by \cite{XJ}, we can get the following inequality
\begin{equation} \label{Eq:3.5}
\int_{0}^{t}\langle t-\tau\rangle^{-\frac{s_{1}+s}{2}} \|(f,\mathcal{G}_{1},\mathcal{G}_{2},\mathcal{G}_{3},\mathcal{G}_{4})(\tau)\|^{\ell}_{\dot{B}^{-s_{1}}_{2,\infty}}d\tau\lesssim \langle t\rangle^{-\frac{s_{1}+s}{2}}\left(\mathcal{D}^{2}_{p}(t)+\mathcal{X}^{2}_{p}(t)\right).
\end{equation}

In order to finish the proof of Proposition \ref{Prop3.1}, it suffices to bound those ``new'' nonlinear terms, which are not available in the barotropic compressible Navier-Stokes system. For that end, let us decompose $\mathcal{G}_{5}$, $\mathcal{G}_{6}$ and $k$ in terms of low-frequency and high frequency as follows:
$$\mathcal{G}_{5}=\mathcal{G}^{\ell}_{5}+\mathcal{G}^{h}_{5}, \ \ \ \ \mathcal{G}_{6}=\mathcal{G}^{\ell}_{6}+\mathcal{G}^{h}_{6}$$
with
$$\mathcal{G}^{\ell}_{5}=-K_{2}(a)\nabla \theta^{\ell}, \ \ \ \ \mathcal{G}^{h}_{5}=-K_{2}(a)\nabla \theta^{h}, \ \ \
\mathcal{G}^{\ell}_{6}=-\theta^{\ell} \nabla K_{3}(a), \ \ \ \ \mathcal{G}^{h}_{6}=-\theta^{h} \nabla K_{3}(a)$$
and
$$k=k^{\ell}+k^{h}$$
with
\begin{eqnarray*}
k^{\ell}&\triangleq&-v \cdot \nabla \theta^{\ell}  -\big(\tilde{K}_{1}(a)+\tilde{K}_{2}(a)\theta\big) \mathrm{div}\, v^{\ell}
+\frac{\tilde{\kappa}'(a) }{(1+a)\nu}\nabla a \cdot \nabla \theta^{\ell} \\
&&+k_{1}(a,\theta^{\ell})+ k_{2}(a,\nabla v,\nabla v^{\ell})\\
k^{h}&\triangleq&-v \cdot \nabla \theta^{h}  -\big(\tilde{K}_{1}(a)+\tilde{K}_{2}(a)\,\theta\big) \mathrm{div}\, v^{h}
+\frac{\tilde{\kappa}'(a)}{(1+a)\nu} \nabla a \cdot \nabla \theta^{h} \\
&&+k_{1}(a,\theta^{h})+ k_{2}(a,\nabla v,\nabla v^{h}),
\end{eqnarray*}
where
\begin{eqnarray*}
	k_{1}(a,\Theta)&\triangleq& \frac{\tilde{\kappa}(a)}{\nu(1+a)} \Delta \Theta -\beta I(a) \Delta \Theta,\\
	k_{2}(a,V_{1},V_{2})&\triangleq& \frac{1}{\nu\chi_{0}}\sqrt{\frac{1}{\mathcal{T}_{\infty}C_{v}}}\left(\frac{\mu\big(\varrho_{\infty}(1+a)\big)}{1+a} \left(V_{1}:V_{2}+V_{1}:{}^T\!V_{2}\right)\right.\\
	&&\left.+\frac{\lambda\big(\varrho_{\infty}(1+a)\big)}{1+a}\mathrm{Tr}\,V_{1} \mathrm{Tr}\,V_{2}\right)
\end{eqnarray*}
and
$$z^{\ell}\triangleq\sum_{j< j_0} \dot{\Delta}_{j}z, \ \ z^{h}\triangleq z-z^{\ell} \ \ \hbox{for} \ \ z=v,\theta.$$

Let us split the proof of Proposition 3.1 into two lemmas.
\begin{lem}\label{Lem3.2}
If $p$ satisfies  \eqref{Eq:1.9}, then it holds that for all $t\geq 0$,
\begin{equation*}
\int_{0}^{t}\langle t-\tau\rangle^{-\frac{s_{1}+s}{2}} \|(\mathcal{G}_{5}^{\ell},\mathcal{G}_{6}^{\ell},k^{\ell})(\tau)\|^{\ell}_{\dot{B}^{-s_{1}}_{2,\infty}}d\tau\lesssim \langle t\rangle^{-\frac{s_{1}+s}{2}}\left(\mathcal{D}^{2}_{p}(t)+\mathcal{X}^{2}_{p}(t)\right),
\end{equation*}
provided that $-s_{1}<s\leq \frac{d}{2}+1$.
\end{lem}
\begin{proof}
Let us first claim that the following three non classical product inequalities
\begin{eqnarray}\label{Eq:3.6}
&&\|FG\|_{\dot{B}^{-s_{1}}_{2,\infty}}
\lesssim\|F\|_{\dot{B}^{\frac{d}{p}}_{p,1}}\|G\|_{\dot{B}^{-s_{1}}_{2,1}},\\
\label{Eq:3.7}
&&\|FG\|_{\dot{B}^{\frac{d}{p}-\frac{d}{2}-s_{1}}_{2,\infty}}
\lesssim\|F\|_{\dot{B}^{\frac{d}{p}-\frac{d}{2}-s_{1}}_{p,1}}\|G\|_{\dot{B}^{\frac{d}{p}}_{2,1}},\\
\label{Eq:3.8}
&&\|FG\|_{\dot{B}^{\frac{d}{p}-\frac{d}{2}-s_{1}}_{2,\infty}}
\lesssim\|F\|_{\dot{B}^{\frac{d}{p}-1}_{p,1}}\|G\|_{\dot{B}^{\frac{d}{p}-\frac{d}{2}-s_{1}+1}_{2,1}}
\end{eqnarray}
for $1-\frac{d}{2}<s_{1}\leq s_{0}$ and $p$ satisfying \eqref{Eq:1.9}. Indeed, the interested reader is referred to \cite{SX,XJ} for the proofs of \eqref{Eq:3.6}-\eqref{Eq:3.7}. It follows from Proposition \ref{Prop2.3} with $\sigma_{1}=\frac{d}{p}-1$, $\sigma_{2}=\frac{d}{p}-\frac{d}{2}-s_{1}+1$, $p_{1}=p$ and  $p_{2}=2$ that
\begin{equation*}
\|FG\|_{\dot{B}^{\frac{d}{p}-\frac{d}{2}-s_{1}}_{2,\infty}}
\lesssim \|F\|_{\dot{B}^{\frac{d}{p}-1}_{p,1}}
\|G\|_{\dot{B}^{\frac{d}{p}-\frac{d}{2}-s_{1}+1}_{2,\infty}}.
\end{equation*}
Hence, \eqref{Eq:3.8} directly stems from the embedding  $\dot{B}^{\frac{d}{p}-\frac{d}{2}-s_{1}+1}_{2,1}\hookrightarrow \dot{B}^{\frac{d}{p}-\frac{d}{2}-s_{1}+1}_{2,\infty}$.

On the other hand, due to Proposition \ref{Prop2.1}, \eqref{Eq:1.13} and the relations $-s_{1}<\frac{d}{2}-1<\frac{d}{2}<\frac{d}{2}+1$ and $\alpha\triangleq s_{1}+\frac{d}{2}+\frac{1}{2}-\varepsilon\geq\frac{s_{1}}{2}+\frac{d}{4}$ for small enough $\varepsilon>0$, we infer that
\begin{equation}\label{Eq:3.9}
 \sup_{\tau \in [0,t]}\langle\tau\rangle^{\frac{s_{1}}{2}+\frac{d}{4}}\|(a,v^{\ell},\theta^{\ell})(\tau)\|_{\dot{B}^{\frac{d}{p}}_{p,1}}+ \sup_{\tau \in [0,t]}\langle\tau\rangle^{\frac{s_{1}}{2}+\frac{d}{4}-\frac{1}{2}}\|(a,v,\theta^{\ell})(\tau)\|_{\dot{B}^{\frac{d}{p}-1}_{p,1}} \lesssim \mathcal{D}_{p}(t)
\end{equation}
and also that, thanks to $-s_{1}<1-s_{1}<2-s_{1}<\frac{d}{2}+1$,
\begin{equation}\label{Eq:3.10}
\|\nabla (a^{\ell},v^{\ell},\theta^{\ell})(\tau)\|_{\dot{B}^{-s_{1}}_{2,1}}
\lesssim  \langle\tau\rangle^{-\frac{1}{2}} \mathcal{D}_{p}(\tau), \ \ \|\nabla^{2}\theta^{\ell}(\tau)\|_{\dot{B}^{-s_{1}}_{2,1}} \lesssim  \langle\tau\rangle^{-1}\mathcal{D}_{p}(\tau) \ \ \hbox{for all} \ \ \tau\geq0.
\end{equation}
Observe that \eqref{Eq:1.13} and the relations $-s_{1}<\frac{d}{2}-1\leq \frac{d}{p}<\frac{d}{p}+1\leq\frac{d}{2}+1$ and $-s_{1}<\frac{d}{p}-\frac{d}{2}-s_{1}+2<\frac{d}{p}+1\leq \frac{d}{2}+1$ for $2 \leq p\leq \frac{2d}{d-2}$ and $s_{1}$ satisfying \eqref{Eq:1.10}, we obviously have
\begin{eqnarray}
\label{Eq:3.11}
&&\|\theta^{\ell}(\tau)\|_{\dot{B}^{\frac{d}{p}}_{2,1}}
\lesssim \langle\tau\rangle^{-\frac{s_{1}}{2}-\frac{d}{2p}} \mathcal{D}_{p}(\tau), \ \ \ \ \|\nabla (v^{\ell},\theta^{\ell})(\tau)\|_{\dot{B}^{\frac{d}{p}}_{2,1}}
\lesssim \langle\tau\rangle^{-\frac{s_{1}}{2}-\frac{d}{2p}-\frac{1}{2}} \mathcal{D}_{p}(\tau), \\
\label{Eq:3.12}
&&\| \nabla (v^{\ell},\theta^{\ell})(\tau)\|_{\dot{B}^{\frac{d}{p}-\frac{d}{2}-s_{1}+1}_{2,1}}\lesssim  \langle\tau\rangle^{-\frac{d}{2p}+\frac{d}{4}-1}\mathcal{D}_{p}(\tau)
\end{eqnarray}
and
\begin{equation}\label{Eq:3.13}
\|(\nabla a^{h},v^{h})(\tau)\|_{\dot{B}^{\frac{d}{p}-1}_{p,1}}+\| \theta^{h}(\tau)\|_{\dot{B}^{\frac{d}{p}-2}_{p,1}}
\lesssim
 \langle\tau\rangle^{-\alpha}\mathcal{D}_{p}(\tau) \ \ \hbox{for all} \ \ \tau\geq0.
\end{equation}

Now, let us begin with proving Lemma \ref{Lem3.2}. To handle the term with $\mathcal{G}_{5}^{\ell}=-K_{2}(a)\nabla \theta^{\ell}$, we write that, thanks to Proposition \ref{Prop2.5} together with \eqref{Eq:3.6}, \eqref{Eq:3.9} and \eqref{Eq:3.10},
\begin{eqnarray*}
\int_{0}^{t}\langle t-\tau\rangle^{-\frac{s_{1}+s}{2}}\|K_{2}(a)\nabla \theta^{\ell}\|^{\ell}_{\dot{B}^{-s_{1}}_{2,\infty}}d\tau
&\lesssim & \int_{0}^{t}\langle t-\tau\rangle^{-\frac{s_{1}+s}{2}}\|a\|_{\dot{B}^{\frac{d}{p}}_{p,1}}
\|\nabla \theta^{\ell}\|_{\dot{B}^{-s_{1}}_{2,1}}d\tau\\
&\lesssim& \mathcal{D}^{2}_{p}(t)\int_{0}^{t}\langle t-\tau\rangle^{-\frac{s_{1}+s}{2}}
\langle \tau\rangle^{-\frac{s_{1}}{2}-\frac{d}{4}-\frac{1}{2}}d\tau.
\end{eqnarray*}
According to $\frac{s_{1}}{2}+\frac{d}{4}+\frac{1}{2}>1$ and $\frac{s_{1}}{2}+\frac{d}{4}+\frac{1}{2}\geq\frac{s_{1}+s}{2}$ for $s_{1}$ satisfying \eqref{Eq:1.10} and $s\leq\frac{d}{2}+1$, inequality \eqref{Eq:3.2} ensures that
\begin{equation*}
\int_{0}^{t}\langle t-\tau\rangle^{-\frac{s_{1}+s}{2}}\|K_{2}(a)\nabla \theta^{\ell}\|^{\ell}_{\dot{B}^{-s_{1}}_{2,\infty}}d\tau
\lesssim \langle t\rangle^{-\frac{s_{1}+s}{2}} \mathcal{D}^{2}_{p}(t).
\end{equation*}
The terms $\tilde{K}_{1}\left( a\right) \mathrm{div}\, v^{\ell}$ and $k_{1}(a,\theta^{\ell})$ (that is, the term $k_{1}(a,\theta)$ is of the type $K(a) \Delta \theta$ with $K(0)=0$)
may be treated at a similar way (use \eqref{Eq:3.6}, \eqref{Eq:3.9}, \eqref{Eq:3.10}, \eqref{Eq:3.2} and Proposition \ref{Prop2.5}), so we feel free to skip them for brevity. Let us decompose
\begin{equation*}
\theta^{\ell} \nabla K_{3}(a)=\theta^{\ell}  K'_{3}(a) \nabla a^{\ell}+\theta^{\ell}  K'_{3}(a) \nabla a^{h} \ \ \hbox{with} \ \ K'_{3}(a)=\chi_{0} \sqrt{\frac{\mathcal{T}_{\infty}}{C_{v}}}\frac{\pi'_{1}\left(\varrho_{\infty}(1+a)\right)}{1+a},
\end{equation*}
\begin{equation*}
v \cdot \nabla \theta^{\ell}=v^{\ell} \cdot \nabla \theta^{\ell}+v^{h} \cdot \nabla \theta^{\ell} \ \ \hbox{and} \ \ \tilde{K}_{2}(a) \,\theta \,\mathrm{div}\, v^{\ell}=\tilde{K}_{2}(a)\theta^{\ell} \mathrm{div}\, v^{\ell}+\tilde{K}_{2}(a)\theta^{h} \mathrm{div}\, v^{\ell}.
\end{equation*}
Regarding the term with $\theta^{\ell} K'_{3}(a) \nabla a^{\ell}$, it follows from Propositions \ref{Prop2.2}, \ref{Prop2.5}, \eqref{Eq:3.6}, \eqref{Eq:3.9}, \eqref{Eq:3.10}, \eqref{Eq:3.3} and \eqref{Eq:3.2} that
\begin{eqnarray*}
&&\int_{0}^{t}\langle t-\tau\rangle^{-\frac{s_{1}+s}{2}}\|\theta^{\ell}  K'_{3}(a) \nabla a^{\ell}\|^{\ell}_{\dot{B}^{-s_{1}}_{2,\infty}}d\tau \\
&\lesssim& \int_{0}^{t}\langle t-\tau\rangle^{-\frac{s_{1}+s}{2}} (1+\|a\|_{\dot{B}^{\frac{d}{p}}_{p,1}}) \|\theta^{\ell}\|_{\dot{B}^{\frac{d}{p}}_{p,1}}
\|\nabla a^{\ell}\|_{\dot{B}^{-s_{1}}_{2,1}} d\tau \\
&\lesssim& \mathcal{D}^{2}_{p}(t) \int_{0}^{t}\langle t-\tau\rangle^{-\frac{s_{1}+s}{2}} \langle \tau\rangle^{-\frac{s_{1}}{2}-\frac{d}{4}-\frac{1}{2}} d\tau
\lesssim \langle t\rangle^{-\frac{s_{1}+s}{2}} \mathcal{D}^{2}_{p}(t).
\end{eqnarray*}
Bounding $v^{\ell} \cdot \nabla \theta^{\ell}$ and $\tilde{K}_{2}\left( a\right) \theta^{\ell} \mathrm{div}\, v^{\ell}$ essentially follows from the same procedure as $\theta^{\ell}  K'_{3}(a) \nabla a^{\ell}$, we thus omit them. To handle the term with $\theta^{\ell} K'_{3}(a) \nabla a^{h}$, we note that, owing to \eqref{Eq:3.7}, \eqref{Eq:3.3} and Propositions \ref{Prop2.2}, \ref{Prop2.5},
\begin{eqnarray*}
\int_{0}^{t}\langle t-\tau\rangle^{-\frac{s_{1}+s}{2}}\|\theta^{\ell}  K'_{3}(a) \nabla a^{h}\|^{\ell}_{\dot{B}^{-s_{1}}_{2,\infty}}d\tau
&\lesssim&
\int_{0}^{t}\langle t-\tau\rangle^{-\frac{s_{1}+s}{2}}\|\theta^{\ell}  K'_{3}(a) \nabla a^{h}\|^{\ell}_{\dot{B}^{\frac{d}{p}-\frac{d}{2}-s_{1}}_{2,\infty}}d\tau\\
&\lesssim& \int_{0}^{t}\langle t-\tau\rangle^{-\frac{s_{1}+s}{2}} \| \theta^{\ell}\|_{\dot{B}^{\frac{d}{p}}_{2,1}} \|\nabla a^{h} \|_{\dot{B}^{\frac{d}{p}-\frac{d}{2}-s_{1}}_{p,1}}d\tau\\
&\lesssim& \int_{0}^{t}\langle t-\tau\rangle^{-\frac{s_{1}+s}{2}} \| \theta^{\ell}\|_{\dot{B}^{\frac{d}{p}}_{2,1}} \|\nabla a^{h} \|_{\dot{B}^{\frac{d}{p}-1}_{p,1}}d\tau,
\end{eqnarray*}
where we used the relations $s_{1}\leq s_{1}+\frac{d}{2}-\frac{d}{p}$ ($p\geq 2$) and $\frac{d}{p}-\frac{d}{2}-s_{1}<\frac{d}{p}-1$ for $s_{1}$ fulfilling \eqref{Eq:1.10}.
According to $\frac{s_{1}}{2}+\frac{d}{2p}+\alpha>\frac{s_{1}}{2}+\frac{d}{4}+\frac{1}{2}>1$ (as $\alpha>\frac{d}{4}+\frac{1}{2}-\frac{d}{2p}$ for sufficiently small $\varepsilon>0$) and $\frac{s_{1}}{2}+\frac{d}{2p}+\alpha >\frac{s_{1}+s}{2}$ for all $s\leq \frac{d}{2}+1$ as well as \eqref{Eq:3.11}, \eqref{Eq:3.13} and \eqref{Eq:3.2}, we arrive at
\begin{eqnarray*}
\int_{0}^{t}\langle t-\tau\rangle^{-\frac{s_{1}+s}{2}}\|\theta^{\ell}  K'_{3}(a) \nabla a^{h}\|^{\ell}_{\dot{B}^{-s_{1}}_{2,\infty}}d\tau
&\lesssim& \mathcal{D}^{2}_{p}(t) \int_{0}^{t}\langle t-\tau\rangle^{-\frac{s_{1}+s}{2}} \langle \tau\rangle^{-\frac{s_{1}}{2}-\frac{d}{2p}-\alpha} d\tau \\
&\lesssim& \langle t\rangle^{-\frac{s_{1}+s}{2}} \mathcal{D}^{2}_{p}(t).
\end{eqnarray*}
For the term $v^{h} \cdot \nabla \theta^{\ell}$,
we take advantage of \eqref{Eq:3.7}, \eqref{Eq:3.11}, \eqref{Eq:3.13}, \eqref{Eq:3.2} and the relations $s_{1}\leq s_{1}+\frac{d}{2}-\frac{d}{p}$ ($p\geq 2$) and $\frac{d}{p}-\frac{d}{2}-s_{1}<\frac{d}{p}-1$ to conclude that we still have
\begin{eqnarray*}
\int_{0}^{t}\langle t-\tau\rangle^{-\frac{s_{1}+s}{2}}\|v^{h} \cdot \nabla \theta^{\ell}\|^{\ell}_{\dot{B}^{-s_{1}}_{2,\infty}} d \tau
&\lesssim& \int_{0}^{t}\langle t-\tau\rangle^{-\frac{s_{1}+s}{2}} \|v^{h}\|_{\dot{B}^{\frac{d}{p}-\frac{d}{2}-s_{1}}_{p,1}}\|\nabla \theta^{\ell}\|_{\dot{B}^{\frac{d}{p}}_{2,1}} d\tau \\
&\lesssim& \int_{0}^{t}\langle t-\tau\rangle^{-\frac{s_{1}+s}{2}} \|v^{h}\|_{\dot{B}^{\frac{d}{p}-1}_{p,1}}\|\nabla \theta^{\ell}\|_{\dot{B}^{\frac{d}{p}}_{2,1}} d\tau \\
&\lesssim& \mathcal{D}^{2}_{p}(t) \int_{0}^{t}\langle t-\tau\rangle^{-\frac{s_{1}+s}{2}} \langle \tau\rangle^{-\frac{s_{1}}{2}-\frac{d}{2p}-\frac{1}{2}-\alpha} d\tau \\
&\lesssim& \langle t\rangle^{-\frac{s_{1}+s}{2}} \mathcal{D}^{2}_{p}(t).
\end{eqnarray*}
To bound the term corresponding to $\tilde{K}_{2}( a)\,\theta^{h} \mathrm{div}\, v^{\ell}$, we observe that applying \eqref{Eq:3.7} with $s_{1}=s_{0}=\frac{2d}{p}-\frac{d}{2}$ yields
\begin{equation}\label{Eq:3.14}
\|FG\|_{\dot{B}^{-\frac{d}{p}}_{2,\infty}}
\lesssim\|F\|_{\dot{B}^{-\frac{d}{p}}_{p,1}}\|G\|_{\dot{B}^{\frac{d}{p}}_{2,1}}.
\end{equation}
Note that $s_{1}\leq s_{0}\leq\frac{d}{p}$ and $-\frac{d}{p}<\frac{d}{p}-2$ ($p<d$), we get from \eqref{Eq:3.14}, \eqref{Eq:3.3} and Propositions \ref{Prop2.2}, \ref{Prop2.5} that
\begin{eqnarray*}
\int_{0}^{t}\langle t-\tau\rangle^{-\frac{s_{1}+s}{2}}\|\tilde{K}_{2}(a)\theta^{h} \mathrm{div}\, v^{\ell}\|^{\ell}_{\dot{B}^{-s_{1}}_{2,\infty}}d\tau
&\lesssim &
\int_{0}^{t}\langle t-\tau\rangle^{-\frac{s_{1}+s}{2}}\|\tilde{K}_{2}(a)\theta^{h} \mathrm{div}\, v^{\ell}\|^{\ell}_{\dot{B}^{-\frac{d}{p}}_{2,\infty}}d\tau\\
&\lesssim & \int_{0}^{t}\langle t-\tau\rangle^{-\frac{s_{1}+s}{2}}\|\theta^{h}\|_{\dot{B}^{-\frac{d}{p}}_{p,1}}
\|\mathrm{div}\, v^{\ell}\|_{\dot{B}^{\frac{d}{p}}_{2,1}}d\tau\\
&\lesssim & \int_{0}^{t}\langle t-\tau\rangle^{-\frac{s_{1}+s}{2}}\|\theta^{h}\|_{\dot{B}^{\frac{d}{p}-2}_{p,1}}
\|\mathrm{div}\, v^{\ell}\|_{\dot{B}^{\frac{d}{p}}_{2,1}}d\tau.
\end{eqnarray*}
In light of \eqref{Eq:3.11}, \eqref{Eq:3.13} and \eqref{Eq:3.2}, we arrive at
\begin{eqnarray*}
\int_{0}^{t}\langle t-\tau\rangle^{-\frac{s_{1}+s}{2}}
\|\tilde{K}_{2}(a)\theta^{h} \mathrm{div}\, v^{\ell}\|^{\ell}_{\dot{B}^{-s_{1}}_{2,\infty}}d\tau
&\lesssim& \mathcal{D}^{2}_{p}(t)\int_{0}^{t}\langle t-\tau\rangle^{-\frac{s_{1}+s}{2}}
\langle \tau\rangle^{-\alpha-\frac{s_{1}}{2}-\frac{d}{2p}-\frac{1}{2}}d\tau\\
&\lesssim& \langle t\rangle^{-\frac{s_{1}+s}{2}} \mathcal{D}^{2}_{p}(t).
\end{eqnarray*}
Let us next look at the term with $\frac{\tilde{\kappa}'(a)}{(1+a)\nu}\nabla a \cdot \nabla \theta^{\ell}$.  Denote by $H(a)$ the smooth function fulfilling $H'(a)=\frac{\tilde{\kappa}'(a)}{(1+a)\nu}$ and $H(0)=0$ so that $\nabla H(a)=\frac{\tilde{\kappa}'(a)}{(1+a)\nu} \nabla a$. So it suffices to estimate the term $\nabla H(a)\cdot \nabla \theta^{\ell}$. With the aid of the fact $s_{1}\leq s_{1}+\frac{d}{2}-\frac{d}{p}$, \eqref{Eq:3.8}, \eqref{Eq:3.9}, \eqref{Eq:3.12}, \eqref{Eq:3.2} and Proposition \ref{Prop2.5}, we end up with
\begin{eqnarray*}
\int_{0}^{t}\langle t-\tau\rangle^{-\frac{s_{1}+s}{2}}\|\nabla H(a) \cdot \nabla \theta^{\ell}\|^{\ell}_{\dot{B}^{-s_{1}}_{2,\infty}} d \tau
&\lesssim& \int_{0}^{t}\langle t-\tau\rangle^{-\frac{s_{1}+s}{2}} \|a\|_{\dot{B}^{\frac{d}{p}}_{p,1}}\|\nabla \theta^{\ell}\|_{\dot{B}^{\frac{d}{p}-\frac{d}{2}-s_{1}+1}_{2,1}} d\tau \\
&\lesssim& \mathcal{D}^{2}_{p}(t) \int_{0}^{t}\langle t-\tau\rangle^{-\frac{s_{1}+s}{2}} \langle \tau\rangle^{-\frac{s_{1}}{2}-\frac{d}{2p}-1} d\tau\\
&\lesssim& \langle t\rangle^{-\frac{s_{1}+s}{2}} \mathcal{D}^{2}_{p}(t),
\end{eqnarray*}
where the relation $p\leq \frac{2d}{d-2}$ ensures $\frac{s_{1}}{2}+\frac{d}{2p}+1\geq \frac{s_{1}}{2}+\frac{d}{4}+\frac{1}{2}>1$ and $\frac{s_{1}}{2}+\frac{d}{2p}+1\geq \frac{s_{1}+s}{2}$ for all $s\leq \frac{d}{2}+1$. We finally decompose $k_{2}(a,\nabla v,\nabla v^{\ell})=k_{2}(a,\nabla v^{\ell},\nabla v^{\ell})+k_{2}(a,\nabla v^{h},\nabla v^{\ell})$.
For the term with $k_{2}(a,\nabla v^{\ell},\nabla v^{\ell})$, it follows from \eqref{Eq:3.8}, \eqref{Eq:3.9}, \eqref{Eq:3.12}, \eqref{Eq:3.3}, \eqref{Eq:3.2}, Propositions \ref{Prop2.2}, \ref{Prop2.5} and the relation $s_{1}\leq s_{1}+\frac{d}{2}-\frac{d}{p}$ that
\begin{eqnarray*}
&&\int_{0}^{t}\langle t-\tau\rangle^{-\frac{s_{1}+s}{2}}\|k_{2}(a,\nabla v^{\ell},\nabla v^{\ell})\|^{\ell}_{\dot{B}^{-s_{1}}_{2,\infty}}d\tau\\
&\lesssim & \int_{0}^{t}\langle t-\tau\rangle^{-\frac{s_{1}+s}{2}}\| v^{\ell}\|_{\dot{B}^{\frac{d}{p}}_{p,1}}\|\nabla v^{\ell}\|_{\dot{B}^{\frac{d}{p}-\frac{d}{2}-s_{1}+1}_{2,1}}d\tau\\
&\lesssim& \mathcal{D}^{2}_{p}(t) \int_{0}^{t}\langle t-\tau\rangle^{-\frac{s_{1}+s}{2}} \langle \tau\rangle^{-\frac{s_{1}}{2}-\frac{d}{2p}-1} d\tau\\
&\lesssim& \langle t\rangle^{-\frac{s_{1}+s}{2}} \mathcal{D}^{2}_{p}(t).
\end{eqnarray*}
Keeping in mind that the relations $s_{1}\leq s_{0}\leq\frac{d}{p}$ and $-\frac{d}{p}<\frac{d}{p}-2$ ($p<d$) and using  \eqref{Eq:3.11}, \eqref{Eq:3.13}, \eqref{Eq:3.14}, \eqref{Eq:3.3}, \eqref{Eq:3.2} and Propositions \ref{Prop2.2}, \ref{Prop2.5}, we conclude that
\begin{eqnarray*}
\int_{0}^{t}\langle t-\tau\rangle^{-\frac{s_{1}+s}{2}}\|k_{3}(a,\nabla v^{h},\nabla v^{\ell})\|^{\ell}_{\dot{B}^{-s_{1}}_{2,\infty}}d\tau
&\lesssim & \int_{0}^{t}\langle t-\tau\rangle^{-\frac{s_{1}+s}{2}}\|\nabla v^{h}\|_{\dot{B}^{-\frac{d}{p}}_{p,1}}
\|\nabla v^{\ell}\|_{\dot{B}^{\frac{d}{p}}_{2,1}}d\tau\\
&\lesssim & \int_{0}^{t}\langle t-\tau\rangle^{-\frac{s_{1}+s}{2}}\|\nabla v^{h}\|_{\dot{B}^{\frac{d}{p}-2}_{p,1}}
\|\nabla v^{\ell}\|_{\dot{B}^{\frac{d}{p}}_{2,1}}d\tau\\
&\lesssim& \mathcal{D}^{2}_{p}(t)\int_{0}^{t}\langle t-\tau\rangle^{-\frac{s_{1}+s}{2}}
\langle \tau\rangle^{-\alpha-\frac{s_{1}}{2}-\frac{d}{2p}-\frac{1}{2}}d\tau \\
&\lesssim& \langle t\rangle^{-\frac{s_{1}+s}{2}} \mathcal{D}^{2}_{p}(t).
\end{eqnarray*}
Hence, putting all estimates together leads to Lemma \ref{Lem3.2}.
\end{proof}
In what follows, let us bound those nonlinear terms in $\mathcal{G}^{h}_{5}$, $\mathcal{G}^{h}_{6}$ and $k^{h}$, precisely
\begin{eqnarray*}
&K_{2}(a)\nabla \theta^{h}, \ \  \theta^{h} \nabla K_{3}(a), \ \ v \cdot \nabla \theta^{h}, \ \ \widetilde{K}_{1}(a) \mathrm{div}\,v^{h}, \\
&\tilde{K}_{2}(a)\theta \mathrm{div}\, v^{h}, \ \ \frac{\tilde{\kappa}'(a)}{(1+a)\nu} \nabla a \cdot \nabla \theta^{h}, \ \ k_{1}(a,\theta^{h}), \ \ k_{2}(a,\nabla v, \nabla v^{h}).
\end{eqnarray*}
In terms of \eqref{Eq:1.13}, we claim that
\begin{equation}\label{Eq:3.15}
\|\theta^{h}(\tau)\|_{\dot{B}^{\frac{d}{p}-1}_{p,1}}+\|(v^{h},\nabla v^{h},\theta^{h})(\tau)\|_{\dot{B}^{\frac{d}{p}}_{p,1}}
\lesssim \tau^{-\alpha}\mathcal{D}_{p}(\tau) \ \ \hbox{for all} \ \ \tau\geq 0.
\end{equation}
\begin{lem}\label{Lem3.3}
	If $p$ satisfies  \eqref{Eq:1.9}, then it holds that for all $t\geq 0$,
	\begin{equation}\label{Eq:3.16}
	\int_{0}^{t}\langle t-\tau\rangle^{-\frac{s_{1}+s}{2}} \|(\mathcal{G}_{5}^{h},\mathcal{G}_{6}^{h},k^{h})(\tau)\|^{\ell}_{\dot{B}^{-s_{1}}_{2,\infty}}d\tau\lesssim \langle t\rangle^{-\frac{s_{1}+s}{2}}\left(\mathcal{D}^{2}_{p}(t)+\mathcal{X}^{2}_{p}(t)\right),
	\end{equation}
provided that $-s_{1}<s\leq \frac{d}{2}+1$.
\end{lem}
\begin{proof}
In order to prove \eqref{Eq:3.16}, we shall present the following inequality
\begin{equation}\label{Eq:3.17}
\|FG^{h}\|^{\ell}_{\dot{B}^{-s_{1}}_{2,\infty}}\lesssim\|FG^{h}\|^{\ell}_{\dot{B}^{-s_{0}}_{2,\infty}}\lesssim\|F\|_{\dot{B}^{\frac{d}{p}-1}_{p,1}}\|G^{h}\|_{\dot{B}^{1-\frac{d}{p}}_{p,1}}\lesssim\|F\|_{\dot{B}^{\frac{d}{p}-1}_{p,1}}\|G^{h}\|_{\dot{B}^{\frac{d}{p}-1}_{p,1}}
\end{equation}
for $1-\frac{d}{2}<s_{1}\leq s_{0}$ and $p$ satisfying \eqref{Eq:1.9}.
The reader is referred to \cite{SX} for the proof of \eqref{Eq:3.17}. To bound the term involving $K_{2}(a)\nabla \theta^{h}$, we see that, thanks to \eqref{Eq:3.17} and Proposition \ref{Prop2.5},
\begin{eqnarray*}
\int_{0}^{t}\langle t-\tau\rangle^{-\frac{s_{1}+s}{2}}\|K_{2}(a)\nabla \theta^{h}\|^{\ell}_{\dot{B}^{-s_{1}}_{2,\infty}}d\tau
&\lesssim & \int_{0}^{t}\langle t-\tau\rangle^{-\frac{s_{1}+s}{2}}\|a\|_{\dot{B}^{\frac{d}{p}-1}_{p,1}}
\|\theta^{h}\|_{\dot{B}^{\frac{d}{p}}_{p,1}}d\tau\\
&=& \left(\int_{0}^{1}+\int_{1}^{t}\right)\left(\cdots\right)d\tau\triangleq I_{1}+I_{2}.
\end{eqnarray*}
It follows from \eqref{Eq:3.1} that
\begin{equation}\label{Eq:3.18}
I_{1} \lesssim \langle t\rangle^{-\frac{s_{1}+s}{2}}\mathcal{X}^{2}_{p}(1)
\end{equation}
and that, owing to \eqref{Eq:3.9}, \eqref{Eq:3.15} and \eqref{Eq:3.2}, if $t\geq 1$,
\begin{eqnarray}\label{Eq:3.19}
I_{2} &\lesssim & \int_{1}^{t}\langle t-\tau\rangle^{-\frac{s_{1}+s}{2}}\langle \tau\rangle^{-\frac{s_{1}}{2}-\frac{d}{4}+\frac{1}{2}-\alpha}\big(\langle \tau\rangle^{\frac{s_{1}}{2}+\frac{d}{4}-\frac{1}{2}}\|a\|_{\dot{B}^{\frac{d}{p}-1}_{p,1}}\big)\big(\tau^{\alpha}\|\theta^{h}\|_{\dot{B}^{\frac{d}{p}}_{p,1}}\big)d\tau \nonumber\\
&\lesssim &(\sup_{\tau \in [1,t]}\langle \tau\rangle^{\frac{s_{1}}{2}+\frac{d}{4}-\frac{1}{2}}\|a\|_{\dot{B}^{\frac{d}{p}-1}_{p,1}}) (\sup_{\tau \in [1,t]}\tau^{\alpha}\|\theta^{h}\|_{\dot{B}^{\frac{d}{p}}_{p,1}})\nonumber \\
&& \times \int_{1}^{t}\langle t-\tau\rangle^{-\frac{s_{1}+s}{2}}
\langle \tau\rangle^{-\frac{s_{1}}{2}-\frac{d}{4}+\frac{1}{2}-\alpha}d\tau
\lesssim  \langle t\rangle^{-\frac{s_{1}+s}{2}}\mathcal{D}^{2}_{p}(t),
\end{eqnarray}
where the fact $\alpha>1$ for small enough $\varepsilon>0$ implies $\frac{s_{1}}{2}+\frac{d}{4}-\frac{1}{2}+\alpha>1$ and $\frac{s_{1}}{2}+\frac{d}{4}-\frac{1}{2}+\alpha\geq \frac{s_{1}+s}{2}$ for $s_{1}$ satisfying \eqref{Eq:1.10} and $s\leq\frac{d}{2}+1$.  Bounding $v \cdot \nabla \theta^{h}$ and $\widetilde{K}_{1}(a) \mathrm{div}\,v^{h}$ essentially follow from the same procedure as $K_{2}(a)\nabla \theta^{h}$, we thus omit them. For the term with $\theta^{h} \nabla K_{3}(a)$, applying \eqref{Eq:3.17} and Proposition \ref{Prop2.5} yields
\begin{eqnarray*}
\int_{0}^{t}\langle t-\tau\rangle^{-\frac{s_{1}+s}{2}}\|\theta^{h} \nabla K_{3}(a)\|^{\ell}_{\dot{B}^{-s_{1}}_{2,\infty}}d\tau
&\lesssim & \int_{0}^{t}\langle t-\tau\rangle^{-\frac{s_{1}+s}{2}}\|a\|_{\dot{B}^{\frac{d}{p}}_{p,1}}
\|\theta^{h}\|_{\dot{B}^{\frac{d}{p}-1}_{p,1}}d\tau\\
&=& \left(\int_{0}^{1}+\int_{1}^{t}\right)\left(\cdots\right)d\tau\triangleq J_{1}+J_{2}.
\end{eqnarray*}
It is clear that $J_{1} \lesssim \langle t\rangle^{-\frac{s_{1}+s}{2}}\mathcal{X}^{2}_{p}(1)$ and that, due to \eqref{Eq:3.9}, \eqref{Eq:3.15} and \eqref{Eq:3.2}, if $t\geq 1$,
\begin{eqnarray*}
J_{2} &\lesssim & \int_{1}^{t}\langle t-\tau\rangle^{-\frac{s_{1}+s}{2}}\langle \tau\rangle^{-\frac{s_{1}}{2}-\frac{d}{4}-\alpha}\big(\langle \tau\rangle^{\frac{s_{1}}{2}+\frac{d}{4}}\|a\|_{\dot{B}^{\frac{d}{p}}_{p,1}}\big)\big(\tau^{\alpha}\|\theta^{h}\|_{\dot{B}^{\frac{d}{p}-1}_{p,1}}\big)d\tau\\
&\lesssim &\big(\sup_{\tau \in [1,t]}\langle \tau\rangle^{\frac{s_{1}}{2}+\frac{d}{4}}\|a\|_{\dot{B}^{\frac{d}{p}}_{p,1}}\big) \big(\sup_{\tau \in [1,t]}\tau^{\alpha}\|\theta^{h}\|_{\dot{B}^{\frac{d}{p}-1}_{p,1}}\big)\int_{1}^{t}\langle t-\tau\rangle^{-\frac{s_{1}+s}{2}}
\langle \tau\rangle^{-\frac{s_{1}}{2}-\frac{d}{4}-\alpha}d\tau\\
&\lesssim & \langle t\rangle^{-\frac{s_{1}+s}{2}}\mathcal{D}^{2}_{p}(t).
\end{eqnarray*}
The term $\frac{\tilde{\kappa}'(a)}{(1+a)\nu} \nabla a \cdot \nabla \theta^{h}$ may be treated at a similar way, so we omit it. Let us look at the term with $\tilde{K}_{2}(a)\theta \mathrm{div}\, v^{h}$. With the aid of \eqref{Eq:3.17}, \eqref{Eq:3.3} and Propositions \ref{Prop2.2}, \ref{Prop2.5}, we arrive at
\begin{eqnarray*}
&&\int_{0}^{t}\langle t-\tau\rangle^{-\frac{s_{1}+s}{2}}\|\tilde{K}_{2}(a)\theta \mathrm{div}\, v^{h}\|^{\ell}_{\dot{B}^{-s_{1}}_{2,\infty}}d\tau
\lesssim \int_{0}^{t}\langle t-\tau\rangle^{-\frac{s_{1}+s}{2}}\|\theta\|_{\dot{B}^{\frac{d}{p}-1}_{p,1}}
\|v^{h}\|_{\dot{B}^{\frac{d}{p}}_{p,1}}d\tau\\
 &\lesssim &\int_{0}^{t}\langle t-\tau\rangle^{-\frac{s_{1}+s}{2}}\big(\|\theta^{\ell}\|_{\dot{B}^{\frac{d}{p}-1}_{p,1}}+\|\theta^{h}\|_{\dot{B}^{\frac{d}{p}-1}_{p,1}}\big)\|v^{h}\|_{\dot{B}^{\frac{d}{p}}_{p,1}}d\tau\\
&=& \left(\int_{0}^{1}+\int_{1}^{t}\right)\left(\cdots\right)d\tau\triangleq H_{1}+H_{2}.
\end{eqnarray*}
It follows from \eqref{Eq:3.1} and the interpolation that $H_{1}\lesssim \langle t\rangle^{-\frac{s_{1}+s}{2}}\mathcal{X}^{2}_{p}(1)$. By using \eqref{Eq:3.9}, \eqref{Eq:3.15} and \eqref{Eq:3.2}, we get, if $t\geq 1$,
\begin{eqnarray*}
H_{2}
&\lesssim &\big(\sup_{\tau \in [1,t]}\langle \tau\rangle^{\frac{s_{1}}{2}+\frac{d}{4}-\frac{1}{2}}\|\theta^{\ell}\|_{\dot{B}^{\frac{d}{p}-1}_{p,1}}+\sup_{\tau \in [1,t]}\tau^{\alpha}\|\theta^{h}\|_{\dot{B}^{\frac{d}{p}-1}_{p,1}}\big) \big(\sup_{\tau \in [1,t]}\tau^{\alpha}\|v^{h}\|_{\dot{B}^{\frac{d}{p}}_{p,1}}\big)\\
&&\times\int_{1}^{t}\langle t-\tau\rangle^{-\frac{s_{1}+s}{2}}
\big(\langle\tau\rangle^{-\frac{s_{1}}{2}-\frac{d}{4}+\frac{1}{2}}+\langle\tau\rangle^{-\alpha}\big)\langle\tau\rangle^{-\alpha} d\tau
\lesssim \langle t\rangle^{-\frac{s_{1}+s}{2}}\mathcal{D}^{2}_{p}(t),
\end{eqnarray*}
where we noticed the fact $\alpha>\frac{s_{1}}{2}+\frac{d}{4}-\frac{1}{2}$ for small enough $\varepsilon>0$.
To bound the term corresponding to $k_{1}(a,\theta^{h})$, it suffices to handle $K(a) \Delta \theta^{h}$ with $K(0)=0$. To this end, one has to consider the cases $2\leq p\leq \frac{2d}{3}$ and $p>\frac{2d}{3}$ separately. If $2\leq p\leq \frac{2d}{3}$, then we have, using \eqref{Eq:3.17} and  Proposition \ref{Prop2.5},
\begin{equation*}
\|K(a) \Delta \theta^{h}\|^{\ell}_{\dot{B}^{-s_{1}}_{2,\infty}}
\lesssim\|a\|_{\dot{B}^{\frac{d}{p}-1}_{p,1}}\|\nabla^{2} \theta^{h}\|_{\dot{B}^{1-\frac{d}{p}}_{p,1}}\lesssim\|a\|_{\dot{B}^{\frac{d}{p}-1}_{p,1}}\|\theta^{h}\|_{\dot{B}^{\frac{d}{p}}_{p,1}},
\end{equation*}
where the fact $2\leq p \leq \frac{2d}{3}$ ensures $1-\frac{d}{p}\leq \frac{d}{p}-2$. By repeating the procedure leading to \eqref{Eq:3.18}-\eqref{Eq:3.19}, we get
\begin{eqnarray*}
\int_{0}^{t}\langle t-\tau\rangle^{-\frac{s_{1}+s}{2}}\|K(a) \Delta \theta^{h}\|^{\ell}_{\dot{B}^{-s_{1}}_{2,\infty}}d\tau &\lesssim & \int_{0}^{t}\langle t-\tau\rangle^{-\frac{s_{1}+s}{2}}\|a\|_{\dot{B}^{\frac{d}{p}-1}_{p,1}}
\|\theta^{h}\|_{\dot{B}^{\frac{d}{p}}_{p,1}}d\tau\\
&\lesssim & \langle t\rangle^{-\frac{s_{1}+s}{2}}\big(\mathcal{X}^{2}_{p}(t)+\mathcal{D}^{2}_{p}(t)\big).
\end{eqnarray*}
Notice that if $p>\frac{2d}{3}$, then applying \eqref{Eq:2.4} with $\sigma=2-\frac{d}{p}>\frac{1}{2}>0$ yields
\begin{equation}\label{Eq:3.20}
\|F G^{h}\|_{\dot {B}^{-s_{1}}_{2,\infty}}^{\ell} \lesssim \|FG^{h}\|_{\dot {B}^{-s_{0}}_{2,\infty}}^{\ell} \lesssim \big(\|F\|_{\dot {B}^{2-\frac{d}{p}}_{p,1}}+\|F^{\ell}\|_{L^{{p}^{*}}}\big)
\|G^{h}\|_{\dot{B}^{\frac{d}{p}-2}_{p,1}} \ \ \hbox{with} \ \ \frac1{p^{*}}\triangleq\frac{1}{2}-\frac{1}{p},
\end{equation}
since $s_{1}\leq s_{0}$. Furthermore, using the composition inequality in Lebesgue spaces and the embeddings $\dot{B}^{\frac{d}{p}}_{2,1}\hookrightarrow L^{p^{*}}$ and $\dot{B}^{s_{0}}_{p,1}\hookrightarrow L^{p^{*}}$ gives
\begin{equation*}
\|K(a)\|_{L^{p^{*}}}\lesssim \|a\|_{L^{p^{*}}}\lesssim \|a^{\ell}\|_{\dot{B}^{\frac{d}{p}}_{2,1}}+\|a^{h}\|_{\dot{B}^{s_{0}}_{p,1}}\lesssim \|a^{\ell}\|_{\dot{B}^{\frac{d}{2}-1}_{2,1}}+\|a^{h}\|_{\dot{B}^{\frac{d}{p}}_{p,1}},
\end{equation*}
where we noticed the relations $\frac{d}{2}-1 \leq \frac{d}{p}$ ($p\leq \frac{2d}{d-2}$) and $s_{0}\leq \frac{d}{p}$. It follows from  Proposition \ref{Prop2.5}, the embedding $\dot{B}^{2-s_{0}}_{2,1} \hookrightarrow \dot{B}^{2-\frac{d}{p}}_{p,1}$ and the relations $2-s_{0}>\frac{d}{2}-1$ ($p>\frac{2d}{3}$) and $2-\frac{d}{p}<\frac{d}{p}$ ($p<d$) that
\begin{equation*}
\|K(a)\|_{\dot{B}^{2-\frac{d}{p}}_{p,1}}\lesssim \|a\|_{\dot{B}^{2-\frac{d}{p}}_{p,1}}\lesssim \|a^{\ell}\|_{\dot{B}^{2-s_{0}}_{2,1}}+\|a^{h}\|_{\dot{B}^{2-\frac{d}{p}}_{p,1}}\lesssim \|a^{\ell}\|_{\dot{B}^{\frac{d}{2}-1}_{2,1}}+\|a^{h}\|_{\dot{B}^{\frac{d}{p}}_{p,1}}.
\end{equation*}
According to \eqref{Eq:3.20}, we deduce that
\begin{eqnarray*}
\int_{0}^{t}\langle t-\tau\rangle^{-\frac{s_{1}+s}{2}}\|k_{1}(a,\theta^{h})\|^{\ell}_{\dot{B}^{-s_{1}}_{2,\infty}}d\tau
&\lesssim&  \int_{0}^{t}\langle t-\tau\rangle^{-\frac{s_{1}+s}{2}}(\|a^{\ell}\|_{\dot{B}^{\frac{d}{2}-1}_{2,1}}+\|a^{h}\|_{\dot{B}^{\frac{d}{p}}_{p,1}})
\|\theta^{h}\|_{\dot{B}^{\frac{d}{p}}_{p,1}}d\tau\\
&=& \left(\int_{0}^{1}+\int_{1}^{t}\right)\left(\cdots\right)d\tau\triangleq \tilde{I}_{1}+\tilde{I}_{2}.
\end{eqnarray*}
With the aid of \eqref{Eq:3.1}, we have $\tilde{I}_{1}\lesssim \langle t\rangle^{-\frac{s_{1}+s}{2}} \mathcal{X}^{2}_{p}(1)$.
Using the fact
\begin{eqnarray*}
\|a^{\ell}(\tau)\|_{\dot{B}^{\frac{d}{2}-1}_{2,1}} \lesssim \langle\tau\rangle^{\frac{s_{1}}{2}+\frac{d}{4}-\frac{1}{2}} \mathcal{D}_{p}(\tau)
\end{eqnarray*}
together with \eqref{Eq:3.13}, \eqref{Eq:3.15} and \eqref{Eq:3.2}, we thus get, if $t\geq1$,
\begin{eqnarray*}
\tilde{I}_{2} &=&  \int_{1}^{t}\langle t-\tau\rangle^{-\frac{s_{1}+s}{2}}(\|a^{\ell}\|_{\dot{B}^{\frac{d}{2}-1}_{2,1}}+\|a^{h}\|_{\dot{B}^{\frac{d}{p}}_{p,1}})
\|\theta^{h}\|_{\dot{B}^{\frac{d}{p}}_{p,1}}d\tau\\
&\lesssim& \big(\sup_{\tau \in [1,t]}\langle\tau\rangle^{\frac{s_{1}}{2}+\frac{d}{4}-\frac{1}{2}}\|a^{\ell}(\tau)\|_{\dot{B}^{\frac{d}{2}-1}_{2,1}}+\sup_{\tau \in [1,t]}\langle\tau\rangle^{\alpha}\|a^{h}(\tau)\|_{\dot{B}^{\frac{d}{p}}_{p,1}}\big)
\big(\sup_{\tau \in [1,t]}\tau ^{\alpha}\|\theta^{h}(\tau)\|_{\dot{B}^{\frac{d}{p}}_{p,1}}\big)\\
&&\times\int_{1}^{t}\langle t-\tau\rangle^{-\frac{s_{1}+s}{2}}
\big(\langle\tau\rangle^{-\frac{s_{1}}{2}-\frac{d}{4}+\frac{1}{2}}+\langle\tau\rangle^{-\alpha}\big)\langle\tau\rangle^{-\alpha} d\tau
\lesssim \langle t\rangle^{-\frac{s_{1}+s}{2}}\mathcal{D}^{2}_{p}(t).
\end{eqnarray*}
Let us finally estimate the term with $k_{2}(a,\nabla v, \nabla v^{h})$. We observe that, owing to \eqref{Eq:3.17}, \eqref{Eq:3.3}, Propositions \ref{Prop2.2}, \ref{Prop2.5} and the interpolation,
\begin{eqnarray*}
&&\int_{0}^{t}\langle t-\tau\rangle^{-\frac{s_{1}+s}{2}}\|k_{3}(a,\nabla v, \nabla v^{h})\|^{\ell}_{\dot{B}^{-s_{1}}_{2,\infty}}d\tau
\lesssim  \int_{0}^{t}\langle t-\tau\rangle^{-\frac{s_{1}+s}{2}}
\|v\|_{\dot{B}^{\frac{d}{p}}_{p,1}}
\|v^{h}\|_{\dot{B}^{\frac{d}{p}}_{p,1}}d\tau\\
&\lesssim&  \int_{0}^{t}\langle t-\tau\rangle^{-\frac{s_{1}+s}{2}}\big(\|v^{\ell}\|_{\dot{B}^{\frac{d}{p}}_{p,1}}+\|v^{h}\|_{\dot{B}^{\frac{d}{p}-1}_{p,1}}\big)
\|\nabla v^{h}\|_{\dot{B}^{\frac{d}{p}}_{p,1}}d\tau\\
&=& \left(\int_{0}^{1}+\int_{1}^{t}\right)\left(\cdots\right)d\tau\triangleq \tilde{J}_{1}+\tilde{J}_{2}.
\end{eqnarray*}
For $\tilde{J}_{1}$, it is clear that $\tilde{J}_{1}\lesssim \langle t\rangle^{-\frac{s_{1}+s}{2}} \mathcal{X}^{2}_{p}(1)$ and that, owing to the fact $\alpha>\frac{s_{1}}{2}+\frac{d}{4}$ for sufficiently small $\varepsilon>0$ as well as \eqref{Eq:3.9}, \eqref{Eq:3.13}, \eqref{Eq:3.15} and \eqref{Eq:3.2}, if $t\geq1$,
\begin{eqnarray*}
\tilde{J}_{2}&= &\int_{1}^{t}\langle t-\tau\rangle^{-\frac{s_{1}+s}{2}}
\big(\|v^{\ell}\|_{\dot{B}^{\frac{d}{p}}_{p,1}}+\|v^{h}\|_{\dot{B}^{\frac{d}{p}-1}_{p,1}}\big)
\|\nabla v^{h}\|_{\dot{B}^{\frac{d}{p}}_{p,1}}d\tau\\
&\lesssim& \big(\sup_{\tau \in [1,t]}\langle\tau\rangle^{\frac{s_{1}}{2}+\frac{d}{4}}\|v^{\ell}\|_{\dot{B}^{\frac{d}{p}}_{p,1}}+\sup_{\tau \in [1,t]}\langle\tau\rangle^{\alpha}\|v^{h}\|_{\dot{B}^{\frac{d}{p}-1}_{p,1}}\big)
(\sup_{\tau \in [1,t]}\tau ^{\alpha}\|\nabla v^{h}\|_{\dot{B}^{\frac{d}{p}}_{p,1}})\\
&&\times\int_{1}^{t}\langle t-\tau\rangle^{-\frac{s_{1}+s}{2}}
\big(\langle\tau\rangle^{-\frac{s_{1}}{2}-\frac{d}{4}}+\langle\tau\rangle^{-\alpha}\big)\langle\tau\rangle^{-\alpha} d\tau
\lesssim \langle t\rangle^{-\frac{s_{1}+s}{2}}\mathcal{D}^{2}_{p}(t).
\end{eqnarray*}
Hence, the proof of Lemma \ref{Lem3.3} is finished.
\end{proof}
Combining \eqref{Eq:3.5} and those estimates in Lemmas \ref{Lem3.2}-\ref{Lem3.3}, we get Proposition \ref{Prop3.1} eventually. Furthermore, with the aid of \eqref{Eq:3.4}, we conclude that
\begin{equation}\label{Eq:3.21}
\langle t\rangle^{\frac{s_{1}+s}{2}}\left\|(a,v,\theta)(t)\right\|^{\ell}_{\dot{B}^{s}_{2,1}}
\lesssim \mathcal{D}_{p,0}+\mathcal{D}^{2}_{p}(t)+\mathcal{X}^{2}_{p}(t) \ \ \hbox{for all} \ \ t\geq0,
\end{equation}
provided that $-s_{1}<s\leq\frac{d}{2}+1$.
\subsection{Second step: decay estimates for the high frequencies of $(\nabla a,v,\theta)$}
In this section, we shall apply the energy method of $L^{p}$ type in terms of the effective velocity. Let $\mathcal{P}\triangleq \mathrm{Id}+\nabla\left(-\Delta\right)^{-1}\mathrm{div}$ be the Leray projector onto divergence-free vector-fields. It follows from \eqref{Eq:1.8} that $\mathcal{P}v$ satisfies
\begin{equation*}
\partial_{t}\mathcal{P}v-\tilde{\mu}_{\infty} \Delta \mathcal{P}v=\mathcal{P}g \ \ \hbox{with} \ \ \tilde{\mu}_{\infty}=\frac{\mu_{\infty}}{\nu}.
\end{equation*}
Let us introduce the effective velocity $w$:
\begin{equation*}
w\triangleq\nabla(-\Delta)^{-1}(a-\mathrm{div}\,v),
\end{equation*}
which was initiated by Hoff \cite{HD} and first used in the context of critical regularity by Haspot \cite{HB} as well as developed by Danchin, the first author and the second auther \cite{DX1,DX2,SX,XJ}. We observe that $(a,w,\theta)$ satisfies
\begin{equation*}
\left\{
\begin{array}{l}
\partial _{t}a+a=f-\mathrm{div}\,w,\\ [2mm]
\partial _{t}w-\Delta w=\nabla(-\Delta )^{-1}(f-\mathrm{div}\,g)-\gamma \nabla \theta+w-(-\Delta )^{-1}\nabla a,\\ [2mm]
\partial_{t}\theta-\beta\Delta\theta=k-\gamma\mathrm{div}\,w-a.
\end{array}
\right.
\end{equation*}
\begin{prop}\label{Prop3.2}
	If $p$ satisfies \eqref{Eq:1.9}, then it holds that for all $T\geq0$,
	\begin{eqnarray}\label{Eq:3.22}
	\|\langle\tau\rangle^{\alpha}(\nabla a,v)\|_{\widetilde{L} ^{\infty}_{T}(\dot {B}^{\frac {d}{p}-1}_{p,1})}^{h}+\|\langle\tau\rangle^{\alpha} \theta \|_{\widetilde{L}^{\infty}_{T}(\dot{B}^{\frac{d}{p}-2}_{p,1})}^{h}
	&\lesssim& \|\left(\nabla a_{0},v_{0}\right)\|^{h}_{\dot{B}_{p,1}^{\frac {d}{p}-1}}+\|\theta_{0}\|^{h}_{\dot{B}_{p,1}^{\frac {d}{p}-2}}\nonumber\\
	&&+\mathcal{D}^{2}_{p}(T)+\mathcal{X}^{2}_{p}(T),
	\end{eqnarray}
with $\alpha=s_{1}+\frac{d}{2}+\frac{1}{2}-\varepsilon$ for sufficiently small $\varepsilon>0$, where $\mathcal{X}_{p}(T)$ and $\mathcal{D}_{p}(T)$ have been defined by \eqref{Eq:3.1} and \eqref{Eq:1.13}, respectively.
\end{prop}
\begin{proof}
By performing the $L^{p}$ energy method, we end up with (see \cite{DX1,DX2,SX,XJ} for details)
\begin{eqnarray}\label{Eq:3.23}
	\|\langle\tau\rangle^{\alpha}(\nabla a,v)\|_{\widetilde{L} ^{\infty}_{T}(\dot {B}^{\frac {d}{p}-1}_{p,1})}^{h}+\|\langle\tau\rangle^{\alpha} \theta \|_{\widetilde{L}^{\infty}_{T}(\dot{B}^{\frac{d}{p}-2}_{p,1})}^{h}
	\lesssim \left\|\left(\nabla a_{0},v_{0}\right)\right\|^{h}_{\dot{B}_{p,1}^{\frac {d}{p}-1}}+\left\|\theta_{0}\right\|^{h}_{\dot{B}_{p,1}^{\frac {d}{p}-2}} \nonumber\\
	+\sum_{j\geq j_{0}-1}\sup_{t\in [0,T]}\left(\langle t\rangle^{\alpha}\int_{0}^{t}e^{-c_{0}(t-\tau)}2^{j(\frac{d}{p}-1)}Z_{j}(\tau)d\tau\right).
\end{eqnarray}
with $Z_{j}\triangleq Z^{1}_{j}+\cdots + Z^{6}_{j}$ and
\begin{eqnarray*}
&&Z^{1}_{j}\triangleq\|\dot{\Delta}_{j}\left(av\right)\|_{L^{p}},\ \ \ \ \ \ \ \ \  \ Z^{2}_{j}\triangleq\|g_{j}\|_{L^{p}},
\ \ \ \ \ \ \ Z^{3}_{j}\triangleq 2^{-j}\|k_{j}\|_{L^{p}},\\
&&Z^{4}_{j}\triangleq\|\nabla \dot{\Delta}_{j}\left(a\,\mathrm{div}\,v\right)\|_{L^{p}}, \ \ \,
Z^{5}_{j}\triangleq\|R_{j}\|_{L^{p}},\ \ \ \ \ \ Z^{6}_{j}\triangleq\|\mathrm{div}\,v\|_{L^{\infty}}\|\nabla a_{j}\|_{L^{p}},
\end{eqnarray*}
where  $a_{j}\triangleq \dot{\Delta}_{j}a$, $g_{j}\triangleq \dot{\Delta}_{j}g$, $k_{j}\triangleq \dot{\Delta}_{j}k$ and $R_{j}\triangleq [u\cdot \nabla,\nabla\dot{\Delta}_{j}]a$.

Firstly, we observe that
\begin{equation*}
\sum_{j\geq j_{0}-1}\sup_{t\in [0,2]}\left(\langle t\rangle^{\alpha}\int_{0}^{t}e^{-c_{0}(t-\tau)}2^{j(\frac{d}{p}-1)}Z_{j}(\tau)d\tau\right)
\lesssim \int_{0}^{2}\sum_{j\geq j_{0}-1}2^{j(\frac{d}{p}-1)}Z_{j}(\tau)d\tau.
\end{equation*}
The terms in $Z^{1}_{j}$, $Z^{4}_{j}$, $Z^{5}_{j}$ and $Z^{6}_{j}$ as well as those in $Z^{2}_{j}$ corresponding to $\mathcal{G}_{1}$, $\mathcal{G}_{2}$, $\mathcal{G}_{3}$ and $\mathcal{G}_{4}$ may be estimated exactly as in \cite{XJ}. Consequently, it is only a matter of handing those ``new" nonlinear terms in $Z^{2}_{j}$ and $Z^{3}_{j}$. Precisely,
\begin{eqnarray*}
&K_{2}(a)\nabla \theta, \ \  \theta \nabla K_{3}(a), \ \ v \cdot \nabla \theta, \ \ \widetilde{K}_{1}(a) \mathrm{div}\,v, \\
&\tilde{K}_{2}(a)\theta \mathrm{div}\, v, \ \ \frac{\tilde{\kappa}'(a)}{(1+a)\nu} \nabla a \cdot \nabla \theta, \ \ k_{1}(a,\theta), \ \ k_{2}(a,\nabla v, \nabla v).
\end{eqnarray*}
To do this, we shall use frequently that, owing to \eqref{Eq:3.1}, interpolation and embeddings (recall that $p\geq 2$),
\begin{equation}\label{Eq:3.24}
\|(a,v)\|_{L_{t}^{2} (\dot{B}_{p,1}^{\frac {d}{p}})} \lesssim \mathcal{X}_{p}(t), \ \ \ \ \ \ \ \ \|(\nabla \theta^{\ell},\theta ^{h})\|_{L^{1}_{t}(\dot{B}^{\frac{d}{p}}_{p,1})}\lesssim \mathcal{X}_{p}(t)
\end{equation}
and also that
\begin{equation}\label{Eq:3.25}
\|(\nabla \theta^{\ell},\theta^{h})\|_{L^{2}_{t}(\dot{B}^{\frac{d}{p}-1}_{p,1})} \lesssim \mathcal{X}_{p}(t),\ \ \ \|(\nabla a,v,\nabla \theta^{\ell})\|_{\tilde{L}^{\infty}_{t}(\dot{B}^{\frac{d}{p}-1}_{p,1})} \lesssim \mathcal{X}_{p}(t).
\end{equation}
For the terms with $K_{2}(a)\nabla \theta$ and $\theta \nabla K_{3}(a)$, we decompose
$$K_{2}(a)\nabla \theta=K_{2}(a)\nabla \theta^{\ell}+K_{2}(a)\nabla \theta^{h} \ \ \hbox{and} \ \ \theta \nabla K_{3}(a)=\theta^{\ell} \nabla K_{3}(a)+\theta^{h} \nabla K_{3}(a).$$
With the aid of Propositions \ref{Prop2.2} and \ref{Prop2.5}, the H\"{o}lder inequality, \eqref{Eq:3.24} and \eqref{Eq:3.25}, we deduce that
\begin{eqnarray*}
\|K_{2}(a)\nabla \theta\|^{h}_{L^{1}_{t}(\dot{B}^{\frac{d}{p}-1}_{p,1})}& \lesssim& \|a\|_{L^{2}_{t}(\dot{B}^{\frac{d}{p}}_{p,1})}\|\nabla \theta^{\ell}\|_{L^{2}_{t}(\dot{B}^{\frac{d}{p}-1}_{p,1})}+\|a\|_{L^{\infty}_{t}(\dot{B}^{\frac{d}{p}}_{p,1})}\|\nabla \theta^{h}\|_{L^{1}_{t}(\dot{B}^{\frac{d}{p}-1}_{p,1})}\\
&\lesssim&  \mathcal{X}^{2}_{p}(t),\\
\|\theta \nabla K_{3}(a)\|^{h}_{L^{1}_{t}(\dot{B}^{\frac{d}{p}-1}_{p,1})}
&\lesssim& \|\theta^{\ell}\|_{L^{2}_{t}(\dot{B}^{\frac{d}{p}}_{p,1})}
\|a\|_{L^{2}_{t}(\dot{B}^{\frac{d}{p}}_{p,1})}+\|\theta^{h}\|_{L^{1}_{t}(\dot{B}^{\frac{d}{p}}_{p,1})}\|a\|_{L^{\infty}_{t}(\dot{B}^{\frac{d}{p}}_{p,1})}\\
&\lesssim& \mathcal{X}^{2}_{p}(t).
\end{eqnarray*}
Keep in mind that the term $\frac{\tilde{\kappa}'(a)}{(1+a)\nu} \nabla a \cdot \nabla \theta$ of $k$ is of the type $\nabla H(a) \cdot \nabla \theta$ with $H(0)=0$, and the term $k_{1}(a,\theta)$ of $k$ is of the type $K(a)\Delta \theta$ with $K(0)=0$. For the terms with $v \cdot \nabla \theta$, $\tilde{K}_{2}(a)\theta \mathrm{div}\, v$, $\frac{\tilde{\kappa}'(a) }{(1+a)\nu}\nabla a \cdot \nabla \theta$ and $k_{1}(a,\theta)$, we decompose them as follows:
$$v \cdot \nabla \theta=v \cdot \nabla \theta^{\ell}+v \cdot \nabla \theta^{h},\ \ \ \ \tilde{K}_{2}(a)\theta \mathrm{div}\, v=\tilde{K}_{2}(a)\theta^{\ell} \mathrm{div}\, v+\tilde{K}_{2}(a)\theta^{h} \mathrm{div}\, v,$$
$$\nabla H(a) \cdot \nabla \theta=\nabla H(a) \cdot \nabla \theta^{\ell}+\nabla H(a) \cdot \nabla \theta^{h}, \ \ \ \ k_{1}(a,\theta)=k_{1}(a,\theta^{\ell})+k_{1}(a,\theta^{h}).$$
Furthermore, we observe that, thanks to  Propositions \ref{Prop2.2} and \ref{Prop2.5} and  \eqref{Eq:3.24}, \eqref{Eq:3.25}, \eqref{Eq:3.3} as well as the relations $p<d$ and $d\geq 3$,
\begin{eqnarray*}
\|v \cdot \nabla \theta\|^{h}_{L^{1}_{t}(\dot{B}^{\frac{d}{p}-2}_{p,1})} &\lesssim&
\|v \cdot \nabla \theta^{\ell}\|^{h}_{L^{1}_{t}(\dot{B}^{\frac{d}{p}-1}_{p,1})}
+\|v \cdot \nabla \theta^{h}\|^{h}_{L^{1}_{t}(\dot{B}^{\frac{d}{p}-2}_{p,1})}\\
 &\lesssim&\|v\|_{L^{2}_{t}(\dot{B}^{\frac{d}{p}}_{p,1})}
\big( \|\nabla \theta^{\ell}\|_{L^{2}_{t}(\dot{B}^{\frac{d}{p}-1}_{p,1})}
 +\|\nabla \theta^{h}\|_{L^{2}_{t}(\dot{B}^{\frac{d}{p}-2}_{p,1})}\big)\\
&\lesssim&  \mathcal{X}^{2}_{p}(t),\\
\|\tilde{K}_{2}(a)\theta \mathrm{div}\,v\|^{h}_{L^{1}_{t}(\dot{B}^{\frac{d}{p}-2}_{p,1})}
&\lesssim&
\|\tilde{K}_{2}(a)\theta^{\ell} \mathrm{div}\,v\|^{h}_{L^{1}_{t}(\dot{B}^{\frac{d}{p}-1}_{p,1})}
+\|\tilde{K}_{2}(a)\theta^{h} \mathrm{div}\,v\|^{h}_{L^{1}_{t}(\dot{B}^{\frac{d}{p}-2}_{p,1})}\\
&\lesssim&\|\theta^{\ell}\|_{L^{2}_{t}(\dot{B}^{\frac{d}{p}}_{p,1})}
\|v\|_{L^{2}_{t}(\dot{B}^{\frac{d}{p}}_{p,1})}+\|\theta^{h}\|_{L^{1}_{t}(\dot{B}^{\frac{d}{p}}_{p,1})}
\|v\|_{L^{\infty}_{t}(\dot{B}^{\frac{d}{p}-1}_{p,1})}\\
&\lesssim&  \mathcal{X}^{2}_{p}(t),\\
\left\|\frac{\tilde{\kappa}'(a)}{(1+a)\nu} \nabla a \cdot \nabla \theta\right\|^{h}_{L^{1}_{t}(\dot{B}^{\frac{d}{p}-2}_{p,1})}
&\lesssim& \|\nabla H(a) \cdot \nabla \theta^{\ell}\|^{h}_{L^{1}_{t}(\dot{B}^{\frac{d}{p}-1}_{p,1})}
+\|\nabla H (a) \cdot \nabla \theta^{h}\|^{h}_{L^{1}_{t}(\dot{B}^{\frac{d}{p}-2}_{p,1})}\\
&\lesssim&\|a\|_{L^{\infty}_{t}(\dot{B}^{\frac{d}{p}}_{p,1})}
\big(\|\nabla \theta^{\ell}\|_{L^{1}_{t}(\dot{B}^{\frac{d}{p}}_{p,1})}+
\|\theta^{h}\|_{L^{1}_{t}(\dot{B}^{\frac{d}{p}}_{p,1})}\big)\\
&\lesssim&  \mathcal{X}^{2}_{p}(t),\\
\|k_{1}(a,\theta)\|^{h}_{L^{1}_{t}(\dot{B}^{\frac{d}{p}-2}_{p,1})}
&\lesssim&\|k_{1}(a,\theta^{\ell})\|^{h}_{L^{1}_{t}(\dot{B}^{\frac{d}{p}-1}_{p,1})}
+\|k_{1}(a,\theta^{h})\|^{h}_{L^{1}_{t}(\dot{B}^{\frac{d}{p}-2}_{p,1})}\\
&\lesssim&\|a\|_{L^{\infty}_{t}(\dot{B}^{\frac{d}{p}}_{p,1})}
\big(\|\nabla \theta^{\ell}\|_{L^{1}_{t}(\dot{B}^{\frac{d}{p}}_{p,1})}+
\|\theta^{h}\|_{L^{1}_{t}(\dot{B}^{\frac{d}{p}}_{p,1})}\big)\\
&\lesssim&  \mathcal{X}^{2}_{p}(t).
\end{eqnarray*}
It follows from Propositions \ref{Prop2.2} and \ref{Prop2.5}, \eqref{Eq:3.3} and  \eqref{Eq:3.24} that
\begin{eqnarray*}
\|\tilde{K}_{1}(a) \,\mathrm{div}\,v\|^{h}_{L^{1}_{t}(\dot{B}^{\frac{d}{p}-2}_{p,1})}
&\lesssim&\|a\|_{L^{2}_{t}(\dot{B}^{\frac{d}{p}}_{p,1})}
\|v\|_{L^{2}_{t}(\dot{B}^{\frac{d}{p}}_{p,1})}
\lesssim \mathcal{X}^{2}_{p}(t),\\
\|k_{2}(a,\nabla v, \nabla v)\|^{h}_{L^{1}_{t}(\dot{B}^{\frac{d}{p}-2}_{p,1})}
&\lesssim&\| v\|^{2}_{L^{2}_{t}(\dot{B}^{\frac{d}{p}}_{p,1})}
\lesssim \mathcal{X}^{2}_{p}(t).
\end{eqnarray*}
Therefore, putting together all the above estimates, we conclude that
\begin{equation}\label{Eq:3.26}
\sum_{j\geq j_{0}-1}\sup_{t\in [0,2]}\left(\langle t\rangle^{\alpha}\int_{0}^{t}e^{-c_{0}(t-\tau)}2^{j(\frac{d}{p}-1)}Z_{j}(\tau)d\tau\right)\lesssim C\mathcal{X}^{2}_{p}(2).
\end{equation}
Secondly, let us bound the supremum for $2\leq t\leq T$ in the last term of \eqref{Eq:3.23}. To this end, one can split the integral on $[0,t]$ into integrals $[0,1]$ and $[1,t]$. The $[0,1]$ part can be bounded exactly as the supremum on $[0,2]$ handled before. In order to deal with the $[1,t]$ part of the integral for $2\leq t\leq T$, we start from
\begin{equation}\label{Eq:3.27}
\sum_{j\geq j_{0}-1}\sup_{t\in [2,T]}\left(\langle t\rangle^{\alpha}\int_{1}^{t}e^{-c_{0}(t-\tau)}2^{j(\frac{d}{p}-1)}Z_{j}(\tau)d\tau\right)\lesssim \sum_{j\geq j_{0}-1} 2^{j\,(\frac{d}{p}-1)}\sup_{t\in [1,T]}t^{\alpha}Z_{j}(t).
\end{equation}
In what follows, we claim the following two inequalities
\begin{eqnarray}\label{Eq:3.28}
&&\|\tau^{\frac{s_{1}}{2}+\frac{d}{4}-\frac{\varepsilon}{2}} a(\tau)\|
_{\tilde{L}_{t}^{\infty}(\dot{B}_{p,1}^{\frac {d}{p}})} \lesssim \mathcal{D}_{p}(t),\\
\label{Eq:3.29}
&&\|\tau^{\frac{s_{1}}{2}+\frac{d}{4}+\frac{m}{2}-\frac{\varepsilon}{2}}\nabla ^{m}(a^{\ell},v^{\ell}, \theta^{\ell})(\tau)\|_{\tilde{L}_{t}^{\infty}(\dot{B}_{p,1}^{\frac {d}{p}})}
\lesssim \mathcal{D}_{p}(t) \ \ \hbox{for} \ \ m=0,1.
\end{eqnarray}
Indeed, it follows from Proposition \ref{Prop2.1}, the fact $\alpha>\frac{s_{1}}{2}+\frac{d}{4}-\frac{\varepsilon}{2}$ for small enough $\varepsilon$, \eqref{Eq:1.13} and tilde norms that
\begin{eqnarray*}
\|\tau^{\frac{s_{1}}{2}+\frac{d}{4}-\frac{\varepsilon}{2}} a(\tau)\|_{\tilde{L}_{t}^{\infty}(\dot{B}_{p,1}^{\frac {d}{p}})}
&\lesssim& \|\tau^{\frac{s_{1}}{2}+\frac{d}{4}-\frac{\varepsilon}{2}} a^{\ell}(\tau)\|_{\tilde{L}_{t}^{\infty}(\dot{B}_{2,1}^{\frac {d}{2}})}
+\|\tau^{\frac{s_{1}}{2}+\frac{d}{4}-\frac{\varepsilon}{2}} a^{h}(\tau)\|_{\tilde{L}_{t}^{\infty}(\dot{B}_{p,1}^{\frac {d}{p}})}\\
&\lesssim& \|\langle \tau\rangle^{\frac{s_{1}}{2}+\frac{d}{4}-\frac{\varepsilon}{2}} a(\tau)\|^{\ell}_{L_{t}^{\infty}(\dot{B}_{2,1}^{\frac {d}{2}-\varepsilon})}
+\|\langle \tau \rangle^{\alpha} a(\tau)\|^{h}_{\tilde{L}_{t}^{\infty}(\dot{B}_{p,1}^{\frac {d}{p}})}\\
&\lesssim& \mathcal{D}_{p}(t),
\end{eqnarray*}
\begin{eqnarray*}
\|\tau^{\frac{s_{1}}{2}+\frac{d}{4}+\frac{m}{2}-\frac{\varepsilon}{2}}\nabla ^{m}(a^{\ell},v^{\ell}, \theta^{\ell})(\tau)\|_{\tilde{L}_{t}^{\infty}(\dot{B}_{p,1}^{\frac {d}{p}})}
&\lesssim& \|\langle \tau\rangle^{\frac{s_{1}}{2}+\frac{d}{4}+\frac{m}{2}-\frac{\varepsilon}{2}}\nabla ^{m}(a^{\ell},v^{\ell}, \theta^{\ell})(\tau)\|_{\tilde{L}_{t}^{\infty}(\dot{B}_{2,1}^{\frac {d}{2}})}\\
&\lesssim& \|\langle \tau \rangle^{\frac{s_{1}}{2}+\frac{d}{4}+\frac{m}{2}-\frac{\varepsilon}{2}}\nabla ^{m}(a,v, \theta)(\tau)\|^{\ell}_{L_{t}^{\infty}(\dot{B}_{2,1}^{\frac {d}{2}-\varepsilon})}\\
&\lesssim& \mathcal{D}_{p}(t) \ \ \hbox{for} \ \ m=0,1.
\end{eqnarray*}
To bound the right-hand side of \eqref{Eq:3.27}, it only need to estimate the ``new'' nonlinear terms (say, $\mathcal{G}_{5}$, $\mathcal{G}_{6}$ and $k$), which are not available in the isentropic compressible Navier-Stokes system, see \cite{XJ} for more details. Regarding the term with $K_{2}(a)\nabla \theta$, we still use the decomposition
$K_{2}(a)\nabla \theta=K_{2}(a)\nabla \theta^{h}+K_{2}(a)\nabla \theta^{\ell}$.
According to Propositions \ref{Prop2.2} and \ref{Prop2.5}, \eqref{Eq:1.13}, \eqref{Eq:3.25}, \eqref{Eq:3.28}, \eqref{Eq:3.29} and tilde norms, we get
\begin{eqnarray*}
\|t^{\alpha}K_{2}(a)\nabla \theta^{h}\|^{h}_{\tilde{L}_{T}^{\infty}(\dot{B}_{p,1}^{\frac{d}{p}-1})}
&\lesssim&\|a\|_{\tilde{L}_{T}^{\infty}(\dot{B}_{p,1}^{\frac {d}{p}})}
\|t^{\alpha} \nabla\theta^{h}\|_{\tilde{L}_{T}^{\infty}(\dot{B}_{p,1}^{\frac{d}{p}-1})}
\lesssim \mathcal{X}_{p}(T)\mathcal{D}_{p}(T),\\
\|t^{\alpha}K_{2}(a)\nabla \theta^{\ell}\|^{h}_{\tilde{L}_{T}^{\infty}(\dot{B}_{p,1}^{\frac{d}{p}-1})}
&\lesssim& \|t^{\frac{s_{1}}{2}+\frac{d}{4}-\frac{\varepsilon}{2}} a\|
_{\tilde{L}_{T}^{\infty}(\dot{B}_{p,1}^{\frac {d}{p}})}
\|t^{\frac{s_{1}}{2}+\frac{d}{4}+\frac{1}{2}-\frac{\varepsilon}{2}} \nabla\theta^{\ell}\|_{\tilde{L}_{T}^{\infty}(\dot{B}_{p,1}^{\frac {d}{p}})}\\
&\lesssim& \mathcal{D}^{2}_{p}(T).
\end{eqnarray*}
For the term with $\theta \nabla K_{3}(a)$, we decompose $\theta \nabla K_{3}(a)=\theta^{h} \nabla K_{3}(a)+\theta^{\ell} \nabla K_{3}(a)$. To handle the term with $\theta^{h} \nabla K_{3}(a)$, we note that, due to Propositions \ref{Prop2.2} and \ref{Prop2.5},  \eqref{Eq:1.13}, \eqref{Eq:3.25} and tilde norms,
\begin{equation*}
\|t^{\alpha}\theta^{h} \nabla K_{3}(a)\|^{h}_{\tilde{L}_{T}^{\infty}(\dot{B}_{p,1}^{\frac{d}{p}-1})}
\lesssim \|t^{\alpha}\theta^{h}\|_{\tilde{L}_{T}^{\infty}(\dot{B}_{p,1}^{\frac{d}{p}})}
\|a\|_{\tilde{L}_{T}^{\infty}(\dot{B}_{p,1}^{\frac {d}{p}})}
\lesssim \mathcal{X}_{p}(T)\mathcal{D}_{p}(T).
\end{equation*}
For the term containing $\theta^{\ell} \nabla K_{3}(a)$, we may write
\begin{equation*}
\theta^{\ell} \nabla K_{3}(a)=\theta^{\ell} K'_{3}(a)\nabla a \ \ \hbox{with} \ \ K'_{3}(a)=\chi_{0} \sqrt{\frac{\mathcal{T}_{\infty}}{C_{v}}} \frac{\pi'_{1}\left(\varrho_{\infty}(1+a)\right)}{1+a}.
\end{equation*}
Now, we have thanks to Propositions \ref{Prop2.2} and \ref{Prop2.5}, \eqref{Eq:1.13}, \eqref{Eq:3.3}, \eqref{Eq:3.25} and \eqref{Eq:3.29},
\begin{eqnarray*}
\|t^{\alpha}\theta^{\ell} K'_{3}(a)\nabla a^{h}\|^{h}_{\tilde{L}_{T}^{\infty}(\dot{B}_{p,1}^{\frac{d}{p}-1})}&\lesssim& \|\theta^{\ell}\|_{\tilde{L}_{T}^{\infty}(\dot{B}_{p,1}^{\frac {d}{p}})}
\|t^{\alpha} \nabla a^{h}\|_{\tilde{L}_{T}^{\infty}(\dot{B}_{p,1}^{\frac{d}{p}-1})}
\lesssim \mathcal{X}_{p}(T)\mathcal{D}_{p}(T),\\
\|t^{\alpha}\theta^{\ell} K'_{3}(a)\nabla a^{\ell}\|^{h}_{\tilde{L}_{T}^{\infty}(\dot{B}_{p,1}^{\frac{d}{p}-1})}&\lesssim& \|t^{\frac{s_{1}}{2}+\frac{d}{4}-\frac{\varepsilon}{2}} \theta^{\ell}\|
_{\tilde{L}_{T}^{\infty}(\dot{B}_{p,1}^{\frac {d}{p}})}
\|t^{\frac{s_{1}}{2}+\frac{d}{4}+\frac{1}{2}-\frac{\varepsilon}{2}} \nabla a^{\ell}\|_{\tilde{L}_{T}^{\infty}(\dot{B}_{p,1}^{\frac {d}{p}})}\\
&\lesssim& \mathcal{D}^{2}_{p}(T).
\end{eqnarray*}
Now, let us keep in mind that the relations $p<d$ and $d\geq 3$. For the term with $v \cdot \nabla \theta$, it follows from \eqref{Eq:1.13}, \eqref{Eq:3.25}, \eqref{Eq:3.29} and Proposition \ref{Prop2.2} adapted to tilde spaces that
\begin{eqnarray*}
\|t^{\alpha}v \cdot \nabla \theta^{h}\|^{h}_{\tilde{L}_{T}^{\infty}(\dot{B}_{p,1}^{\frac{d}{p}-2})}
&\lesssim&\|v\|_{\tilde{L}_{T}^{\infty}(\dot{B}_{p,1}^{\frac {d}{p}-1})}
\|t^{\alpha} \nabla\theta^{h}\|_{\tilde{L}_{T}^{\infty}(\dot{B}_{p,1}^{\frac{d}{p}-1})}
\lesssim \mathcal{X}_{p}(T)\mathcal{D}_{p}(T),\\
\|t^{\alpha}v^{h} \cdot \nabla \theta^{\ell}\|^{h}_{\tilde{L}_{T}^{\infty}(\dot{B}_{p,1}^{\frac{d}{p}-2})}
&\lesssim& \|t^{\alpha} v^{h}\|
_{\tilde{L}_{T}^{\infty}(\dot{B}_{p,1}^{\frac {d}{p}})}
\|\nabla\theta^{\ell}\|_{\tilde{L}_{T}^{\infty}(\dot{B}_{p,1}^{\frac {d}{p}-1})}
\lesssim \mathcal{X}_{p}(T)\mathcal{D}_{p}(T),\\
\|t^{\alpha}v^{\ell} \cdot \nabla \theta^{\ell}\|^{h}_{\tilde{L}_{T}^{\infty}(\dot{B}_{p,1}^{\frac{d}{p}-2})}
&\lesssim& \|t^{\frac{s_{1}}{2}+\frac{d}{4}-\frac{\varepsilon}{2}} v^{\ell}\|
_{\tilde{L}_{T}^{\infty}(\dot{B}_{p,1}^{\frac {d}{p}})}
\|t^{\frac{s_{1}}{2}+\frac{d}{4}+\frac{1}{2}-\frac{\varepsilon}{2}} \nabla\theta^{\ell}\|_{\tilde{L}_{T}^{\infty}(\dot{B}_{p,1}^{\frac {d}{p}})}\\
&\lesssim& \mathcal{D}^{2}_{p}(T),
\end{eqnarray*}
With the aid of Propositions \ref{Prop2.2} and \ref{Prop2.5}, \eqref{Eq:1.13}, \eqref{Eq:3.25}, \eqref{Eq:3.28} and \eqref{Eq:3.29}, we arrive at
\begin{eqnarray*}
\|t^{\alpha}\tilde{K}_{1}(a) \mathrm{div}\,v^{h}\|^{h}_{\tilde{L}_{T}^{\infty}(\dot{B}_{p,1}^{\frac{d}{p}-2})}
&\lesssim& \|a\|_{\tilde{L}_{T}^{\infty}(\dot{B}_{p,1}^{\frac {d}{p}})}
\|t^{\alpha} \mathrm{div}\,v^{h}\|_{\tilde{L}_{T}^{\infty}(\dot{B}_{p,1}^{\frac{d}{p}-2})}
\lesssim \mathcal{X}_{p}(T)\mathcal{D}_{p}(T),\\
\|t^{\alpha}\tilde{K}_{1}(a) \mathrm{div}\,v^{\ell}\|^{h}_{\tilde{L}_{T}^{\infty}(\dot{B}_{p,1}^{\frac{d}{p}-2})}
&\lesssim & \|t^{\frac{s_{1}}{2}+\frac{d}{4}-\frac{\varepsilon}{2}} a\|
_{\tilde{L}_{T}^{\infty}(\dot{B}_{p,1}^{\frac {d}{p}})}
\|t^{\frac{s_{1}}{2}+\frac{d}{4}+\frac{1}{2}-\frac{\varepsilon}{2}} \mathrm{div}\,v^{\ell}\|_{\tilde{L}_{T}^{\infty}(\dot{B}_{p,1}^{\frac {d}{p}})}\\
&\lesssim& \mathcal{D}^{2}_{p}(T),
\end{eqnarray*}
To deal with the term containing $\tilde{K}_{2}(a)\theta \,\mathrm{div}\, v$, we observe that, thanks to Propositions \ref{Prop2.2} and \ref{Prop2.5}, \eqref{Eq:1.13}, \eqref{Eq:3.3}, \eqref{Eq:3.25} and \eqref{Eq:3.29} that
\begin{eqnarray*}
\|t^{\alpha}\tilde{K}_{2}(a)\theta^{h} \mathrm{div}\, v\|^{h}_{\tilde{L}_{T}^{\infty}(\dot{B}_{p,1}^{\frac{d}{p}-2})}
&\lesssim&\|t^{\alpha}\theta^{h}\|_{\tilde{L}_{T}^{\infty}(\dot{B}_{p,1}^{\frac {d}{p}})}
\|\mathrm{div}\, v\|_{\tilde{L}_{T}^{\infty}(\dot{B}_{p,1}^{\frac{d}{p}-2})}
\lesssim \mathcal{X}_{p}(T)\mathcal{D}_{p}(T),\\
\|t^{\alpha}\tilde{K}_{2}(a)\theta^{\ell} \mathrm{div}\, v^{h}\|^{h}_{\tilde{L}_{T}^{\infty}(\dot{B}_{p,1}^{\frac{d}{p}-2})}
&\lesssim& \|\theta^{\ell}\|
_{\tilde{L}_{T}^{\infty}(\dot{B}_{p,1}^{\frac {d}{p}})}
\|t^{\alpha} \mathrm{div}\, v^{h}\|_{\tilde{L}_{T}^{\infty}(\dot{B}_{p,1}^{\frac {d}{p}})}
\lesssim \mathcal{X}_{p}(T)\mathcal{D}_{p}(T),\\
\|t^{\alpha}\tilde{K}_{2}(a)\theta^{\ell} \mathrm{div}\, v^{\ell}\|^{h}_{\tilde{L}_{T}^{\infty}(\dot{B}_{p,1}^{\frac{d}{p}-2})}
&\lesssim &\|t^{\frac{s_{1}}{2}+\frac{d}{4}-\frac{\varepsilon}{2}} \theta^{\ell}\|
_{\tilde{L}_{T}^{\infty}(\dot{B}_{p,1}^{\frac {d}{p}})}
\|t^{\frac{s_{1}}{2}+\frac{d}{4}+\frac{1}{2}-\frac{\varepsilon}{2}} \mathrm{div}\, v^{\ell}\|_{\tilde{L}_{T}^{\infty}(\dot{B}_{p,1}^{\frac {d}{p}})}\\
&\lesssim& \mathcal{D}^{2}_{p}(T).
\end{eqnarray*}
Recall that the term $\frac{\tilde{\kappa}'(a)}{(1+a)\nu} \nabla a \cdot \nabla \theta$ of $k$ is of the type $\nabla H(a) \cdot \nabla \theta$ with $H(0)=0$, and the term $k_{1}(a,\theta)$ of $k$ is of the type $K(a) \Delta \theta$ with $K(0)=0$. Consequently, from  Propositions \ref{Prop2.2} and \ref{Prop2.5}, \eqref{Eq:1.13}, \eqref{Eq:3.25}, \eqref{Eq:3.28} and \eqref{Eq:3.29},  we infer that
\begin{eqnarray*}
&&\|t^{\alpha}\nabla H(a) \cdot \nabla \theta^{h}\|^{h}_{\tilde{L}_{T}^{\infty}(\dot{B}_{p,1}^{\frac{d}{p}-2})}
+\|t^{\alpha}K(a) \Delta \theta^{h}\|^{h}_{\tilde{L}_{T}^{\infty}(\dot{B}_{p,1}^{\frac{d}{p}-2})}\\
&\lesssim& \|a\|_{\tilde{L}_{T}^{\infty}(\dot{B}_{p,1}^{\frac {d}{p}})}
\|t^{\alpha} \theta^{h}\|_{\tilde{L}_{T}^{\infty}(\dot{B}_{p,1}^{\frac{d}{p}})}
\lesssim \mathcal{X}_{p}(T)\mathcal{D}_{p}(T),\\
&&\|t^{\alpha}\nabla H(a) \cdot \nabla \theta^{\ell}\|^{h}_{\tilde{L}_{T}^{\infty}(\dot{B}_{p,1}^{\frac{d}{p}-2})}
+\|t^{\alpha}K(a) \Delta \theta^{\ell}\|^{h}_{\tilde{L}_{T}^{\infty}(\dot{B}_{p,1}^{\frac{d}{p}-2})}\\
&\lesssim& \|t^{\frac{s_{1}}{2}+\frac{d}{4}-\frac{\varepsilon}{2}} a\|_{\tilde{L}_{T}^{\infty}(\dot{B}_{p,1}^{\frac {d}{p}})}
\|t^{\frac{s_{1}}{2}+\frac{d}{4}+\frac{1}{2}-\frac{\varepsilon}{2}} \nabla \theta^{\ell}\|_{\tilde{L}_{T}^{\infty}(\dot{B}_{p,1}^{\frac{d}{p}})}
\lesssim \mathcal{D}^{2}_{p}(T).
\end{eqnarray*}
To bound the term containing $k_{2}(a,\nabla v, \nabla v)$, we take advantage of Propositions \ref{Prop2.2} and \ref{Prop2.5}, \eqref{Eq:1.13}, \eqref{Eq:3.3}, \eqref{Eq:3.25} and \eqref{Eq:3.29}, and get
\begin{eqnarray*}
\|t^{\alpha}k_{2}(a,\nabla v, \nabla v^{h})\|^{h}_{\tilde{L}_{T}^{\infty}(\dot{B}_{p,1}^{\frac{d}{p}-2})}
&\lesssim& \|\nabla v\|_{\tilde{L}_{T}^{\infty}(\dot{B}_{p,1}^{\frac {d}{p}-2})}
\|t^{\alpha} \nabla v^{h}\|_{\tilde{L}_{T}^{\infty}(\dot{B}_{p,1}^{\frac{d}{p}})}
\lesssim \mathcal{X}_{p}(T)\mathcal{D}_{p}(T),\\
\|t^{\alpha}k_{2}(a,\nabla v^{h}, \nabla v^{\ell})\|^{h}_{\tilde{L}_{T}^{\infty}(\dot{B}_{p,1}^{\frac{d}{p}-2})}
&\lesssim& \|t^{\alpha} \nabla v^{h}\|_{\tilde{L}_{T}^{\infty}(\dot{B}_{p,1}^{\frac{d}{p}})}
\|\nabla v^{\ell}\|_{\tilde{L}_{T}^{\infty}(\dot{B}_{p,1}^{\frac {d}{p}-2})}
\lesssim \mathcal{X}_{p}(T)\mathcal{D}_{p}(T),\\
\|t^{\alpha}k_{2}(a,\nabla v^{\ell}, \nabla v^{\ell})\|^{h}_{\tilde{L}_{T}^{\infty}(\dot{B}_{p,1}^{\frac{d}{p}-2})}
&\lesssim& \|t^{\frac{s_{1}}{2}+\frac{d}{4}-\frac{\varepsilon}{2}}\nabla v^{\ell}\|_{\tilde{L}_{T}^{\infty}(\dot{B}_{p,1}^{\frac{d}{p}-1})}
\|t^{\frac{s_{1}}{2}+\frac{d}{4}+\frac{1}{2}-\frac{\varepsilon}{2}}\nabla v^{\ell}\|_{\tilde{L}_{T}^{\infty}(\dot{B}_{p,1}^{\frac {d}{p}})}\\
&\lesssim& \mathcal{D}^{2}_{p}(T).
\end{eqnarray*}
Putting all the above estimates together, we discover that
\begin{equation}\label{Eq:3.30}
\sum_{j\geq j_{0}-1} 2^{j\,(\frac{d}{p}-1)}\sup_{t\in [1,T]}t^{\alpha}Z_{j}(t)
\lesssim\mathcal{X}_{p}(T)\mathcal{D}_{p}(T)+\mathcal{D}^{2}_{p}(T).
\end{equation}
Plugging \eqref{Eq:3.30} in \eqref{Eq:3.27}, and remembering \eqref{Eq:3.26} and \eqref{Eq:3.23}, we end up with \eqref{Eq:3.22}. This completes the proof of Proposition \ref{Prop3.2}.
\end{proof}
\subsection{Third step: Decay and gain of regularity for the high frequencies of $(\nabla v, \theta)$}
Let us prove that the parabolic smoothing effect provided by the last two equations of \eqref{Eq:1.8} allows us to get gain of regularity and decay altogether for $v$ and $\theta$. Precisely, one has
\begin{prop}\label{Prop3.3}
	If $p$ satisfies \eqref{Eq:1.9}, then it holds that for all $t\geq0$,
	\begin{equation}\label{Eq:3.31}
	\left\|\tau ^{\alpha}(\nabla v, \theta)\right\|^{h}_{\tilde{L}^{\infty}_{t}(\dot{B}^{\frac{d}{p}}_{p,1})}
	\lesssim \|(\nabla a_{0},v_{0})\|^{h}_{\dot{B}_{p,1}^{\frac {d}{p}-1}}+\|\theta_{0}\|^{h}_{\dot{B}_{p,1}^{\frac {d}{p}-2}}+\mathcal{X}^{2}_{p}(t)+\mathcal{D}^{2}_{p}(t)
	\end{equation}
with $\alpha=s_{1}+\frac{d}{2}+\frac{1}{2}-\varepsilon$ for sufficiently small $\varepsilon>0$, where $\mathcal{X}_{p}(t)$ and $\mathcal{D}_{p}(t)$ have been defined by \eqref{Eq:3.1} and \eqref{Eq:1.13}, respectively.
\end{prop}
\begin{proof}
It follows from the second and third equations in \eqref{Eq:1.8} that
	\begin{equation}\label{Eq:3.32}
	\left\{
	\begin{array}{l}
	\partial _{t}v-\tilde{\mathcal{A}}v=g-\nabla a-\gamma \nabla \theta,\\[2mm]
	\partial_{t}\theta-\beta\Delta\theta=k-\gamma\mathrm{div}\,v.
	\end{array}
	\right.
	\end{equation}
In order to prove \eqref{Eq:3.31}, we reformulate \eqref{Eq:3.32} as follows
\begin{equation*}
	\left\{
	\begin{array}{l}
	\partial_{t}(t^{\alpha}\tilde{\mathcal{A}}v)-\tilde{\mathcal{A}}(t^{\alpha}\tilde{\mathcal{A}}v)= t^{\alpha}\tilde{\mathcal{A}}g+\alpha t^{\alpha-1}\tilde{\mathcal{A}}v-t^{\alpha}\tilde{\mathcal{A}}\nabla a-\gamma t^{\alpha}\tilde{\mathcal{A}}\nabla \theta,\\[2mm]
	\partial_{t}(t^{\alpha}\Delta \theta)-\beta \Delta(t^{\alpha}\Delta \theta)= t^{\alpha}\Delta k+\alpha t^{\alpha-1}\Delta\theta-\gamma t^{\alpha}\Delta \mathrm{div}\,v,\\ [2mm]
	(t^{\alpha}\tilde{\mathcal{A}}v, t^{\alpha}\Delta \theta)|_{t=0}=(0,0).
	\end{array}
	\right.
	\end{equation*}
Taking advantage of Proposition \ref{Prop2.7}, Remark \ref{Rem2.2} and Bernstein inequality, we have for $j\geq j_{0}-1$,
\begin{eqnarray*}
\|\tau^{\alpha}\nabla v\|^{h}_{\tilde{L}^{\infty}_{t}(\dot{B}^{\frac{d}{p}}_{p,1})}
&\lesssim&\|\tau^{\alpha}g\|^{h}_{\tilde{L}^{\infty}_{t}(\dot{B}^{\frac{d}{p}-1}_{p,1})}+\|\tau^{\alpha-1}v\|^{h}_{\tilde{L}^{\infty}_{t}(\dot{B}^{\frac{d}{p}-1}_{p,1})}
+\|\tau^{\alpha} \nabla a\|^{h}_{\tilde{L}^{\infty}_{t}(\dot{B}^{\frac{d}{p}-1}_{p,1})}\\
&&+\|\tau^{\alpha}\theta\|^{h}_{\tilde{L}^{\infty}_{t}(\dot{B}^{\frac{d}{p}}_{p,1})},
\\
\|\tau^{\alpha}\theta\|^{h}_{\tilde{L}^{\infty}_{t}(\dot{B}^{\frac{d}{p}}_{p,1})}
&\lesssim&\|\tau^{\alpha}k\|^{h}_{\tilde{L}^{\infty}_{t}(\dot{B}^{\frac{d}{p}-2}_{p,1})}+\|\tau^{\alpha-1}\theta\|^{h}_{\tilde{L}^{\infty}_{t}(\dot{B}^{\frac{d}{p}-2}_{p,1})}+\|\tau^{\alpha} v\|^{h}_{\tilde{L}^{\infty}_{t}(\dot{B}^{\frac{d}{p}-1}_{p,1})}.
\end{eqnarray*}
As $\alpha>1$ for small enough $\varepsilon>0$, we see that
\begin{eqnarray*}
\|\tau^{\alpha-1}(\nabla v, \theta)\|^{h}_{\tilde{L}^{\infty}_{t}(\dot{B}^{\frac{d}{p}-2}_{p,1})}
&\lesssim&  \|\langle\tau\rangle^{\alpha}(\nabla v, \theta)\|^{h}_{\tilde{L}^{\infty}_{t}(\dot{B}^{\frac{d}{p}-2}_{p,1})},\\
\|\tau^{\alpha} (\nabla a,v)\|^{h}_{\tilde{L}^{\infty}_{t}(\dot{B}^{\frac{d}{p}-1}_{p,1})}
&\lesssim& \|\langle\tau\rangle^{\alpha} (\nabla a, v)\|^{h}_{\tilde{L}^{\infty}_{t}(\dot{B}^{\frac{d}{p}-1}_{p,1})}.
\end{eqnarray*}
Furthermore, we deduce that
\begin{eqnarray} \label{Eq:3.33}
\|\tau^{\alpha} (\nabla v, \theta)\|^{h}_{\tilde{L}^{\infty}_{t}(\dot{B}^{\frac{d}{p}}_{p,1})}
&\lesssim& \|\tau^{\alpha} g\|^{h}_{\tilde{L}^{\infty}_{t}(\dot{B}^{\frac{d}{p}-1}_{p,1})}
+\|\tau^{\alpha} k\|^{h}_{\tilde{L}^{\infty}_{t}(\dot{B}^{\frac{d}{p}-2}_{p,1})} \nonumber\\
&&+\left\|\langle\tau\rangle^{\alpha}(\nabla a,v)\right\|^{h}_{\tilde{L}^{\infty}_{t}(\dot{B}^{\frac{d}{p}-1}_{p,1})}
+\|\langle\tau\rangle^{\alpha}\theta\|^{h}_{\tilde{L}^{\infty}_{t}(\dot{B}^{\frac{d}{p}-2}_{p,1})}.
\end{eqnarray}
With the aid of \eqref{Eq:3.22}, the last two norms of the r.h.s of \eqref{Eq:3.33} can be bounded by
	$$\|(\nabla a_{0},v_{0})\|^{h}_{\dot{B}_{p,1}^{\frac {d}{p}-1}}+\|\theta_{0}\|^{h}_{\dot{B}_{p,1}^{\frac {d}{p}-2}}+\mathcal{D}^{2}_{p}(t)+\mathcal{X}^{2}_{p}(t).$$
Bounding the norms $\left\|\tau^{\alpha}g\right\|^{h}_{\tilde{L}^{\infty}_{t}(\dot{B}^{\frac{d}{p}-1}_{p,1})}$ and $\|\tau^{\alpha} k\|^{h}_{\tilde{L}^{\infty}_{t}(\dot{B}^{\frac{d}{p}-2}_{p,1})}$ are exactly same as the second step and those work of \cite{XJ}, one can conclude that \eqref{Eq:3.31} readily.
\end{proof}
Finally, adding up \eqref{Eq:3.31} to \eqref{Eq:3.22} and  \eqref{Eq:3.21} yields for all $T\geq0$,
\begin{equation*}
\mathcal{D}_{p}(T)\lesssim \mathcal{D}_{p,0}
+\|(\nabla a_{0},v_{0})\|^{h}_{\dot{B}_{p,1}^{\frac {d}{p}-1}}+\|\theta_{0}\|^{h}_{\dot{B}_{p,1}^{\frac {d}{p}-2}}+\mathcal{X}^{2}_{p}(T)+\mathcal{D}^{2}_{p}(T).
\end{equation*}
The global existence result (see for example Theorem 1.1 in \cite{DX2}) ensures that $\mathcal{X}_{p}(t)\lesssim\mathcal{X}_{p,0}\ll1$
and as
\begin{equation*}
\|(a_{0},v_{0},\theta_{0})\|^{\ell}_{\dot{B}^{\frac{d}{2}-1}_{2,1}}\lesssim \|(a_{0},v_{0},\theta_{0})\|^{\ell}_{\dot{B}^{-s_{1}}_{2,\infty}},
\end{equation*}
one can conclude that \eqref{Eq:1.12} is satisfied for all time if $\mathcal{D}_{p,0}$ and $\mathcal{X}_{p,0}$
are small enough. This completes the proof of Theorem \ref{Thm1.1}.
\subsection{The proof of Corollary \ref{Cor1.1}}
\begin{proof}
It is suffices to show the decay estimate for $\theta$. With the aid of the embedding $\dot{B}_{p,1}^{s+d\,(\frac{1}{p}-\frac{1}{r})}\hookrightarrow\dot{B}^{s}_{r,1}$ for $p\leq r\leq\infty$, we arrive at
\begin{eqnarray*}
	&&\sup_{t\in[0,T]} t^{-\frac{s_{1}+s}{2}-\frac{d}{2}(\frac{1}{2}-\frac{1}{r})}\|\Lambda ^{s}\theta\|_{\dot{B}^{0}_{r,1}}\lesssim\sup_{t\in[0,T]} t^{-\frac{s_{1}+s}{2}-\frac{d}{2}(\frac{1}{2}-\frac{1}{r})}\|\theta\|_{\dot{B}^{s}_{r,1}}\\
&\lesssim& \sup_{t\in[0,T]} t^{-\frac{s_{1}+s}{2}-\frac{d}{2}(\frac{1}{2}-\frac{1}{r})}\|\theta\|_{\dot{B}_{p,1}^{s+d\,(\frac{1}{p}-\frac{1}{r})}}\\
&\lesssim&\| t^{-\frac{s_{1}+s}{2}-\frac{d}{2}(\frac{1}{2}-\frac{1}{r})}\theta
\|^{\ell}_{L^{\infty}_{T}(\dot{B}^{s+d\,(\frac{1}{2}-\frac{1}{r})}_{2,1})}
	+\| t^{-\frac{s_{1}+s}{2}-\frac{d}{2}(\frac{1}{2}-\frac{1}{r})}\theta\|^{h}_{L^{\infty}_{T}(\dot{B}^{s+d\,(\frac{1}{p}-\frac{1}{r})}_{p,1})}.
\end{eqnarray*}
Thanks to \eqref{Eq:1.12} and \eqref{Eq:1.13}, we discover that
\begin{eqnarray*}
\| t^{-\frac{s_{1}+s}{2}-\frac{d}{2}(\frac{1}{2}-\frac{1}{r})}\theta
\|^{\ell}_{L^{\infty}_{T}(\dot{B}^{s+d\,(\frac{1}{2}-\frac{1}{r})}_{2,1})}
&\lesssim& \| \langle t\rangle^{-\frac{s_{1}+s}{2}-\frac{d}{2}(\frac{1}{2}-\frac{1}{r})}\theta
\|^{\ell}_{L^{\infty}_{T}(\dot{B}^{s+d\,(\frac{1}{2}-\frac{1}{r})}_{2,1})} \\
&\lesssim&
 \big(\mathcal{D}_{p,0}+\left\|\left(\nabla a_{0},v_{0}\right)\right\|^{h}_{\dot{B}_{p,1}^{\frac {d}{p}-1}}+\left\|\theta_{0}\right\|^{h}_{\dot{B}_{p,1}^{\frac {d}{p}-2}}\big),
\end{eqnarray*}
where we used the fact $-s_{1}<s+d\,(\frac{1}{2}-\frac{1}{r}) \leq \frac{d}{2}+1$ for $-\tilde{s}_{1}<s+d\,(\frac{1}{p}-\frac{1}{r})\leq\frac{d}{p}+1$.
On the other hand, if $\varepsilon>0$ is small enough, then we have
$ \frac{s_{1}}{2}+\frac{d}{2}+\frac{1}{2}-\varepsilon>1+\frac{d}{4}-\varepsilon\geq\frac{d}{4}
\geq \frac{s}{2}+\frac{d}{2}(\frac{1}{2}-\frac{1}{r})$ for $-\tilde{s}_{1}<s+d\,(\frac{1}{p}-\frac{1}{r})\leq\frac{d}{p}$, which ensures that $\alpha\geq \frac{s_{1}+s}{2}+\frac{d}{2}(\frac{1}{2}-\frac{1}{r})$. Consequently, we deduce that
	\begin{eqnarray*}
\|t^{-\frac{s_{1}+s}{2}-\frac{d}{2}(\frac{1}{2}-\frac{1}{r})}\theta\|^{h}_{L^{\infty}_{T}(\dot{B}^{s+d\,(\frac{1}{p}-\frac{1}{r})}_{p,1})}
	&\lesssim&\| t^{-\frac{s_{1}+s}{2}-\frac{d}{2}(\frac{1}{2}-\frac{1}{r})}\theta\|^{h}_{\tilde{L}^{\infty}_{T}(\dot{B}^{s+d\,(\frac{1}{p}-\frac{1}{r})}_{p,1})}\\
	&\lesssim& \big(\mathcal{D}_{p,0}+\left\|\left(\nabla a_{0},v_{0}\right)\right\|^{h}_{\dot{B}_{p,1}^{\frac {d}{p}-1}}+\left\|\theta_{0}\right\|^{h}_{\dot{B}_{p,1}^{\frac {d}{p}-2}}\big).
	\end{eqnarray*}
Hence, using $\dot{B}^{0}_{r,1}\hookrightarrow L^{r}$ yields the desired result for $\theta$. Proving the inequalities for $a$ and $v$ is similar. The proof of Corollary \ref{Cor1.1} is complete.
\end{proof}

\section*{Acknowledgments}
The first author is supported by the Nanjing University of Aeronautics and Astronautics PhD short-term visiting scholar project (181004DF08).
The second author (J. Xu) is grateful to Professor R. Danchin for addressing the conjecture on the regularity of low frequencies when visiting the LAMA in UPEC. His research is partially supported by the National
Natural Science Foundation of China (11471158, 11871274) and the Fundamental Research Funds for the Central
Universities (NE2015005).

\end{document}